\documentclass[reqno,11pt,a4paper]{amsart}

\usepackage{jc}
\usepackage{tikz-cd}
\usepackage{pgfplots}
\pgfplotsset{compat=1.18}

\newcommand{\DD}{\mathbb{D}}

\newcommand{\twSp}{\widetilde{\on{Sp}}}

\newcommand{\Sp}{\on{Sp}}
\newcommand{\SL}{\on{SL}}
\newcommand{\twSL}{\widetilde{\on{SL}}}

\newcommand{\pp}[2]{\frac{\partial#1}{\partial#2}}

\newcommand{\norm}[1]{\left\lVert#1\right\rVert}

\newcommand{\TopEn}{\on{En}}
\newcommand{\En}{\on{En}}
\newcommand{\Var}{\on{Var}}

\newcommand{\cH}{\mathcal{H}}

\newcommand{\cI}{\mathcal{I}}

\newcommand{\SDiff}{\operatorname{SDiff}}

\newcommand{\RR}{\mathbb{R}}

\begin{document}

\author{Robert Cardona}
\address{Departament de Matematiques i Inform\`{a}tica\\Universitat de
Barcelona\\ Gran Via de les Corts Catalanes 585\\08007 Barcelona\\Spain;
Centre de Recerca Matemàtica\\ Campus de Bellaterra, Edifici C\\ 08193, Barcelona\\ Spain}
\email{robert.cardona@ub.edu}

\author{Julian Chaidez}
\address{Department of Mathematics\\University of Southern California\\Los Angeles, CA\\90007\\USA}
\email{julian.chaidez@usc.edu}

\author{Francisco Torres de Lizaur}
\address{Departamento de Analisis Matem\'{a}tico \& IMUS\\ Universidad de Sevilla\\ C/ Tarfia S/N\\41012 Sevilla\\Spain}
\email{ftorres2@us.es}

\title{On dynamical invariants of coadjoint orbits of 3D volume-preserving diffeomorphisms}

\begin{abstract} 
The helicity, or asymptotic linking number, is a functional of exact volume-preserving vector fields on 3-manifolds, invariant under volume-preserving transformations. It is known to exhibit remarkable uniqueness properties: many invariant functionals reduce to functions of helicity. We examine how severely this uniqueness can fail. On integral homology spheres with the $C^{1}$-topology, the failure is extreme: for every $C^{1}$-open set of nonvanishing exact fields of fixed helicity, some other global dynamical invariant is continuous and non-constant in that set; the Ruelle invariant if some field is non-Anosov, and topological entropy otherwise. In particular, on the three-sphere, the Ruelle invariant is everywhere independent of helicity. This implies, in a very strong sense, a negative answer to a question of Arnold and Khesin on the somewhere density of coadjoint orbits of 3D volume-preserving diffeomorphisms, when considered for the $C^1$-topology. On arbitrary three-manifolds, we also answer the question in the negative using local rather than global invariants.
\end{abstract}

\maketitle

\section{Introduction} 

Let $M$ be a 3-manifold endowed with a smooth volume form $\mu$. A vector field $X$ on $M$ is volume-preserving if and only if the 2-form $\omega = \iota_X\mu$ is closed. A volume-preserving vector field is called \emph{exact}, or null-homologous, if this 2-form is exact. The set of smooth exact vector fields on the 3-manifold, denoted by
\[\mathfrak{X}_\mu^0(M)\]
is a vector space on which the group of smooth volume preserving diffeomorphism $\SDiff_{\mu}(M)$ acts by pushforward, $X \rightarrow \Phi_{*} X$. If $M$ is simply connected, this action can be understood as the coadjoint representation of the group $G=\SDiff_{\mu}(M)$ on its dual Lie algebra $\mathfrak{g}^{*}$, which can be identified with $\mathfrak{X}_\mu^0(M)$ (see \cite[Section I.7.B]{ak}). If $M$ is not simply connected, then $\mathfrak{X}_\mu^0(M)$ is the dual Lie algebra of the subgroup of diffeomorphisms generated by exact vector fields. 

Non-vanishing exact vector fields can be thought of as diffuse links on $M$, and their orbits under the volume-preserving group, as generalizations of isotopy classes of knots and links. This viewpoint was emphasized by Arnold and Khesin \cite[ Section VI.3]{ak}, who argued that the challenging classification of knot and link invariants can be seen as part of the (rather poorly understood) problem on the classification of functionals $\cI: \mathfrak{X}_\mu^0(M)\longrightarrow \mathbb{R}$ which are invariant by the coadjoint action, namely, they satisfy $\cI(\Phi_{*} X)=\cI(X)$ (defined either among non-vanishing exact fields or all exact fields).

Moreover, understanding these coadjoint orbits and their invariants is key in the study of hydrodynamical and magnetohydrodynamical phenomena, because of Helmholtz's and Alfven's conservation theorems (we refer again to Arnold and Khesin's 1998 monograph on the subject \cite{ak}, especially to Section I.9, pages 50-51, where the description of coadjoint orbits is presented as an "unsolved and intriguing problem"). For instance, the equations of ideal magnetohydrodynamics, that model the motion of a conducting fluid in the limit of zero resistivity, have the remarkable property that the magnetic field $B_t$ created by the fluid current (which is a time-dependent vector field, usually assumed to be exact) remains "frozen-in'' in the fluid, that is, $B_t=(\Phi_{t})_{*}B_0$, with $\Phi_{t}$ the path of volume-preserving diffeomorphisms describing the trajectories of the fluid particles. Along with this, in the process known as magnetic relaxation (see \cite{mo1985}), the $L^2$ norm of realistic solutions is expected to decrease over time, until some equilibrium state is reached. Thus one is led to the following variational problem to find equilibria: minimize the $L^2$-norm over the coadjoint orbits. Since the coadjoint orbit is a very complex object, a natural approach is to minimize energy subject to some functional invariants being fixed. 

The paradigmatic example of such an invariant is the \emph{helicity}. Given any primitive 1-form $\alpha$ of $\omega$, which satisfies $d\alpha =\omega = \iota_X\mu$, the helicity is given by the integral
\[
\cH(X)=\int_{M} \omega \wedge \alpha\,,
\]
which does not depend on the primitive by Stokes theorem. Helicity was first introduced by Woltjer \cite{wo1958} in the context of astrophysics and independently by Moffat \cite{mo1969} in hydrodynamics. It is easy to check that the helicity of a vector field supported on a collection of linked tubes is proportional to the linking number of the tubes. More generally, Arnold showed that helicity is equal to the asymptotic linking number between distinct pairs of trajectories of the field \cite{ar1986} (see also \cite{vo2003}). 

Helicity plays a preeminent role in the study of 3D flows. The Woltjer-Taylor scenario in magnetohydrodynamics (see e.g. \cite{Kom22} and references therein) consist precisely in replacing, in the variational problem above, the whole coadjoint orbit by the level set of the helicity to which the initial magnetic field belongs: Woltjer (\cite{wo1958}, see also \cite{al1991}) showed that minimizers are then Beltrami fields, while Taylor \cite{Tay86} argued that, in a realistic situation where the flow is turbulent and there is a small resistivity, Beltrami fields should be the expected equilibria, because helicity should be the only invariant functional that remains approximately conserved. 

In \cite[Section I.9]{ak}, it was actually conjectured that the helicity is unique among functionals that are invariant along coadjoint orbits and satisfy a technical condition (being the "integral of a local density"). This conjecture was proven for $C^1$-vector fields in \cite{EPSTL} (the first result was proven in \cite{Kud1} for flows with cross-sections in manifolds with boundary, see also \cite{Kud2}) and later \cite{KPSY} for $C^k$-vector fields with $k\geq 4$ or $k=\infty$.

\subsection{Main results.} Motivated by the study of invariants of coadjoint orbits and the uniqueness properties of helicity, in this work, we address the following natural question.

\begin{question*} How much does the uniqueness of helicity fail if the condition of being the integral of a local density is removed? \end{question*}

An extreme situation of the failure of this uniqueness would be that it fails everywhere locally: in every open set of vector fields with given helicity, there is some other global invariant that is continuous and non-trivial in this open set. In the best case, one could even hope for a global invariant that is everywhere continuous and independent of helicity. In any of these two situations, every coadjoint orbit is constrained within helicity level sets. Our first main result establishes this extreme situation for non-vanishing vector fields on integral homology spheres with the $C^1$-topology. The global invariants that we use are the Ruelle invariant (which is everywhere continuous, see Proposition \ref{prop:Ru_continuity}) and the topological entropy (which is continuous \cite{An} and even differentiable \cite{katok1990differentiability} on any open set that contains only Anosov flows).

\begin{thm*}\label{thm:main}
    Let $M$ be an integral homology sphere. Fix any $h\in \mathbb{R}$ and any $C^1$-open set $\mathcal{U}$ in the set of non-vanishing exact fields with helicity $h$. According to two cases, the following holds:
    \begin{enumerate}
        \item If $\mathcal{U}$ contains some non-Anosov vector field, then the Ruelle invariant is non-constant in $\mathcal{U}$.
        \item If $\mathcal{U}$ only contains Anosov vector fields, then the topological entropy is non-constant in $\mathcal{U}$.
    \end{enumerate}
\end{thm*}

\begin{remark*}\label{rem:sphere}
Our result contributes to the comparative study of helicity and the Ruelle invariant, initiated by Gambaudo and Ghys \cite{GG}, see also Ghys’ ICM plenary lecture \cite{G2007}. In particular, it shows that the Ruelle invariant is $C^1$-everywhere independent of helicity among non-vanishing volume preserving flows on the 3-sphere, which admits no Anosov flows.
\end{remark*}

The strategy to prove the two items in Theorem \ref{thm:main} can be summarized as follows. When an exact vector field is non-Anosov, we show that it can be $C^1$-perturbed, fixing its helicity, so that it admits a compact invariant set (a solid torus) where the dynamics are of a specific simple model. For this model, a toric flow-tube (\ref{def:toric_flow_tube}), we compute the (relative) helicity and the Ruelle invariant, and use this computation to show that it is possible to modify the latter while keeping fixed the former by arbitrarily small perturbations. The second item in Theorem \ref{thm:main} relies on the use of the formulas for the derivatives of the topological entropy of Anosov flows \cite{katok1990differentiability, katok1991formulas, pollicott1994derivatives}, which are used to find appropriate perturbations that fix helicity but not the entropy.
\begin{remark*}
For the proof of Theorem \ref{thm:main}, we make use of \cite[Theorem 1]{BeDu}, in particular of the $C^1$-density of conservative flows with elliptic periodic orbits among non-Anosov flows, and adapt it to exact fields with fixed helicity. When doing this adaptation, we noticed a gap in the proof of that theorem, due to the use of a "pasting" lemma in \cite{AM}, which was detected to be wrong in \cite{Tei}. We explain how to fix this gap in Section \ref{ss:hellipticorbits}, for which we need to prove a Franks-type lemma for conservative flows of independent interest (see Lemma \ref{lem:franks}, whose proof is deferred to the Appendix). This lemma corrects related results \cite{BD}, and possibly others like \cite{BR} that relied on the (incorrect) pasting lemma.
\end{remark*}
Theorem \ref{thm:main}, and Theorem \ref{thm:main2} below, contribute to the understanding of a conjecture of Arnold and Khesin \cite[p 50]{ak} and its relation to helicity uniqueness.
\begin{conjecture*} [Arnold-Khesin 1998]\label{conj:AK}
There exists a vector field $X$ whose coadjoint orbit is somewhere dense in its corresponding helicity level set, for some topology.
\end{conjecture*}
A positive answer to this conjecture precludes the existence of a set of invariants for which, on any given open set of exact fields, at least one invariant is continuous and independent of helicity. Theorem \ref{thm:main} shows that such a set of invariants exists on integral homology spheres in the $C^1$-topology, and on the three-sphere, there is even an everywhere continuous invariant independent of helicity. The lack of an analog of the Ruelle invariant on an arbitrary three-manifold precludes us from generalizing Theorem \ref{thm:main} to an arbitrary ambient manifold. However, our second result establishes its consequence concerning Arnold-Khesin's conjecture on arbitrary three-manifolds. 
\begin{thm*}\label{thm:main2}
Let $M$ be a closed three-manifold equipped with a volume form $\mu$, and fix some $h\in \mathbb{R}$. For any $X\in \mathfrak{X}_\mu^0(h)$, the coadjoint orbit
$$\mathcal{O}(X)= \{ \varphi_*X\mid \varphi\in \SDiff_\mu(M)\} $$
is nowhere dense in $\mathfrak{X}_\mu^0(h)$ in the $C^1$-topology.
\end{thm*}
\begin{remark*}\label{rem:notcoadj}
The set $\mathcal{O}(X)$ above is technically an adjoint orbit when $M$ is simply connected, but we stick to the "coadjoint orbit" nomenclature (see Remark \ref{rem:coadj}). When $M$ is not simply connected, the set $\mathcal{O}(X)$ above is not a coadjoint orbit anymore, but contains the coadjoint orbit of the subgroup of diffeomorphisms generated by exact fields, $\SDiff_\mu^0(M)$, see Section \ref{sec:back} below. However, it is the natural set to consider in view of the motivation coming from hydrodynamics and magnetohydrodynamics, where the transported exact vector field evolves by pushforward of non-necessarily exact diffeomorphisms.
\end{remark*}

The proof of Theorem \ref{thm:main2} is inspired by our first result, where the goal was to find coadjoint invariant dynamical quantities that are continuous and independent of helicity. In contrast with Theorem \ref{thm:main}, to prove Theorem \ref{thm:main2} it is enough to work with continuous invariants of "local" nature, both in the sense that they are only defined in a sufficiently small neighborhood of a given coadjoint orbit, and in the sense that these invariants are measured locally in the manifold on certain robust zeroes or periodic orbits of the vector field. The idea of considering invariants localized at zeroes and periodic orbits is already suggested in Arnold-Khesin's monograph, see \cite[Remark 9.4 (B)]{ak}. Local invariants of zeroes are considered in \cite[Remark 3]{Kh2024} in the context of the non-mixing properties of the Euler equations, see also further works \cite{KKPS, KKPS2, CTdL23}. It turns out that such local invariants can be defined and proven to be continuous and non-trivial for an open and dense set of coadjoint orbits, which is enough to deduce Theorem \ref{thm:main2}. Both Theorem \ref{thm:main} and Remark \ref{rem:sphere} are significantly stronger versions of Theorem \ref{thm:main2} below. \\

\subsection*{Organization of the paper}
Section \ref{sec:back} fixes notation and provides background on exact volume preserving vector fields, their coadjoint orbits, and helicity. In Section \ref{sec:perturbations} we adapt some genericity results for vector fields to the case of vector fields with \emph{fixed} helicity that are needed in the rest of the paper: these include a version of Kupka-Smale theorem and Pugh's closing lemma, as well as a version of  \cite[Theorem 1]{BeDu}. Along the way, we state a Franks-type lemma that fills some gaps in the literature on 3D conservative flows, the proof of this lemma can be found in Appendix \ref{app}. 
In Section \ref{sec:Ruelle} we establish the necessary properties of the Ruelle invariant, including explicit computations of the invariant for vector fields on toric flow tubes, how to construct these by arbitrarily small perturbations, and prove the first item in Theorem \ref{thm:main}. Section \ref{sec:Anosov} is concerned with the topological entropy of Anosov flows and proves the second item in Theorem \ref{thm:main}. Finally, in Section \ref{sec:orbits} we prove Theorem \ref{thm:main2}.\\

\subsection*{Acknowledgements} The authors are grateful to Theodore Drivas and Oliver Edtmair for useful discussions related to the topic of this paper. We also thank Mário Bessa for helpful correspondence concerning the density of elliptic periodic orbits and the Franks-type lemma. RC acknowledges partial support from the AEI grant PID2023-147585NA-I00, the Departament de Recerca i Universitats de la Generalitat de Catalunya (2021 SGR 00697), and the Spanish State Research Agency, through the Severo Ochoa and María de Maeztu Program for Centers and Units of Excellence in R\&D (CEX2020-001084-M). JC acknowledges partial support from the NSF grant DMS-2446019. FTL acknowledges support from the Spanish Research Agency's grant PID2022-140494NA-I00 and the Ram\'on y Cajal contract RYC2021-034872. 

\section{Background}\label{sec:back}

Let $M$ be a closed three-dimensional manifold endowed with a volume form $\mu \in \Omega^3(M)$. As we have mentioned in the introduction, an important subclass of volume-preserving vector fields is that of exact fields.
\begin{definition}
    A vector field $X\in \mathfrak{X}(M)$ is \emph{volume-preserving} if $\mathcal{L}_X\mu=0$, or equivalently, if $\iota_X\mu$ is a closed two-form. We say that $X$ is \emph{exact} if $\iota_X\mu$ is an exact two-form. 
\end{definition}

\begin{notation} 
In this paper, we use the following notation for certain spaces of vector fields. Let
\[\mathfrak{X}_\mu(M) \qquad\text{and}\qquad \mathfrak{X}_\mu^0(M)\]
denote the sets of volume-preserving and exact vector fields of $C^\infty$-regularity on $(M,\mu)$. Similarly, we denote the analogous sets of vector fields of $C^k$-regularity, with $k\geq 1$, by
\[\mathfrak{X}_{\mu,k}(M) \qquad\text{and}\qquad \mathfrak{X}^{0}_{\mu,k}(M).\]
\end{notation}
When $M$ is a rational homology sphere, any volume-preserving vector field is exact. We denote by $\SDiff_\mu^0(M)$ the set of exact volume-preserving diffeomorphisms, namely, a diffeomorphism
$$\varphi: M \longrightarrow M$$ 
belongs to $\SDiff_\mu^0$ if and only if there is a time-dependent vector field $X_t\in \mathfrak{X}_\mu^0(M)$ such that $\varphi$ is the time-one map of the flow generated by $X_t$. Endow the group of volume-preserving diffeomorphisms $\SDiff_\mu(M)$ (all diffeomorphisms $\varphi:M\rightarrow M$ such that $\varphi^*\mu=\mu$) with the composition operation. As a Lie group, it acts on $\mathfrak{X}_\mu(M)$ by pushforward
$$\varphi_*(X)=D\varphi \circ X\circ \varphi^{-1}.$$
It induces an action on $\mathfrak{X}_\mu^0(M)$, in other words, the vector field $\varphi_*(X)$ is exact whenever $X$ is exact. When $M$ is simply connected, and thus $\mathfrak{X}_\mu(M)=\mathfrak{X}_\mu^0(M)$, the action can be identified with the coadjoint action
\begin{align*}
    \operatorname{Ad}_{\varphi}^*: \mathfrak{X}_\mu^0(M) &\longrightarrow \mathfrak{X}_\mu^0(M)\\
    X &\longmapsto \varphi_*(X)
\end{align*}
of $\SDiff_\mu(M)$ on its dual Lie algebra. 

\begin{remark}\label{rem:coadj}
In the fluid mechanics literature, the space of diffeomorphic vorticities, which by Helmholtz's theorem is the natural phase space for the Euler equation, is often called the co-adjoint orbit, hence our choice of terminology. Strictly speaking, $\varphi_*X$ is the \emph{adjoint} action of the group $\SDiff_\mu^0(M)$ on its Lie algebra.  The dual Lie algebra of $\SDiff_\mu(M)$ is identified with the set of one-forms modulo exact forms \cite[Example 3.12 and Theorem 7.5]{ak}, but (in the simply connected case) the latter can be identified with $\mathfrak{X}_\mu^0(M)$ by associating a one-form $\alpha$ (concretely, an equivalence class of one-forms) with the volume-preserving vector field determined by the equation $\iota_X\mu=d\alpha$.
\end{remark}

\begin{definition}
Given $X\in \mathfrak{X}_\mu^0(M)$, the set
$$\mathcal{O}(X)=\{\varphi_*X\mid \varphi\in \SDiff_\mu(M)\} \subset \mathfrak{X}_\mu^0(M)$$
is called the \emph{coadjoint orbit}\footnote{As mentioned in Remark \ref{rem:notcoadj}, this is only a coadjoint orbit when $M$ is simply connected. However, we will keep this terminology for any three-manifold.} of $X$.
\end{definition}

For exact fields, one can define the quantity known as helicity.
\begin{definition}
    Given $X\in \mathfrak{X}_\mu^0(M)$, and $\alpha$ a primitive of $\iota_X\mu$, we define the helicity of $X$ as
    $$\mathcal{H}(X):=\int_M \iota_X\mu\wedge \alpha.$$
\end{definition}
The helicity functional
$$\mathcal{H}: \mathfrak {X}_\mu^0(M) \longrightarrow \mathbb{R}$$
is continuous when we endow $\mathfrak{X}_\mu^0(M)$ with the $C^k$-topology, with $k\geq 0$. Its value is constant along coadjoint orbits. Indeed, given $\varphi\in \SDiff_\mu(M)$ and $X\in\mathfrak{X}_\mu^0(M)$, denote by $\tilde \alpha$ a primitive of $\iota_{\varphi_*X}\mu$. Then, we have
\begin{align*}
    \mathcal{H}(\varphi_*X)&= \int_M \iota_{\varphi_*X}\mu\wedge \tilde \alpha\\
                &= \int_{\varphi(M)} \iota_{\varphi_*X}\mu\wedge \tilde \alpha\\
                &= \int_M \varphi^*(\iota_{\varphi_*X}\mu\wedge \tilde \alpha)\\
                &= \int_M \iota_X\varphi^*\mu\wedge \varphi^*\tilde \alpha\\
                &=\mathcal{H}(X),
\end{align*}
where in the last equality we have used that $\varphi$ is volume-preserving and that $\varphi^*\tilde \alpha$ is a primitive of $\varphi^*(\iota_{\varphi_*X}\mu)=\iota_X\mu$.

\begin{notation} We will denote by $\mathfrak{X}_\mu^0(h)\subset \mathfrak{X}_\mu^0(M)$ the subset of smooth exact vector fields with given helicity $h\in \mathbb{R}$, and when the ambient manifold is a homology sphere, we can simply denote it by $\mathfrak{X}_\mu(h)$. For lower regularity vector fields, we use $\mathfrak{X}_{\mu,k}^0(h)$. \end{notation}

\section{Perturbations with fixed helicity}\label{sec:perturbations}

In this section, we introduce some perturbation results for vector fields in $\mathfrak{X}_\mu^0(h)$, for any value $h$. These involve a helicity correction lemma and an adaptation of perturbation and genericity results for volume-preserving vector fields to this setting. We will also need to prove a Franks-type lemma (Lemma \ref{lem:franks}), of independent interest, required to correct the proof of \cite[Theorem 1]{BeDu}. The statement and adaptation to fixed helicity of \cite[Theorem 1]{BeDu} is needed for the proof of Theorem \ref{thm:main}. This slightly technical section can be safely skipped on a first reading, pausing only to note the main statements needed for later use, Corollary \ref{cor:localmod} and Theorem \ref{thm:fixedhelliptic}.

\subsection{Helicity corrections}\label{ss:helcorrec}
Let us describe a simple way of continuously increasing or decreasing the helicity of a vector field using a compactly supported perturbation near a point $p$. Endow $M$ with a volume form $\mu$ and an exact volume-preserving vector field $X$ of $C^k$-regularity ($k\geq 1$ or $k=\infty$). Let $U$ be any open subset of $M$ where $X$ does not vanish. Consider the set of functions
$$C^{\infty}_{c,-1}(U)= \{f\in C^\infty(U) \mid \operatorname{supp}{f}\subset U, f>-1\}.$$ 
Since $X$ is exact, we can fix any primitive $\alpha$ such that $\iota_X\mu=d\alpha$, and we can assume without loss of generality that it satisfies $\alpha\wedge d\alpha\neq 0$ in $U$. We denote by $\mathfrak{X}_\mu^0(U,X)$ the set of exact volume-preserving vector fields in $U$ that coincide with $X$ away from some (not fixed) open set whose closure is strictly contained in $U$. We define a map
   \begin{align*}
       P_X: C^\infty_{c,-1}(U) &\longrightarrow \mathfrak{X}_{\mu}^0(U,X)\\
       f &\longmapsto Y
   \end{align*}
    where $Y=P_X(f)$ is uniquely defined by the equation 
    $$\iota_{P_X(f)}\mu = d(\sqrt{(1+f)}\alpha).$$
    This map is continuous when we endow $C^\infty_c(U)$ with the $C^\infty$ topology and $\mathfrak{X}_{\mu}^0(M)$ with the $C^{k}$ topology, and depends on the choice of primitive $\alpha$. The helicity of the vector field $P_X(f)$ is given by
    \begin{align*}
        \mathcal{H}(P_X(f))&= \int_M (1+f) \alpha\wedge d\alpha\\
        &= \mathcal{H}(X) + \int_U f\alpha \wedge d\alpha.
    \end{align*}
    Since $\mathcal{H}$ is continuous with the $C^1$-topology, the map $\mathcal{H}\circ P_X$ is continuous in the $C^k$ topology. Having fixed some $\delta$, we can always choose non-negative or non-positive non-trivial functions in $\{f\in C^\infty_{c,-1}(U)\mid \norm{f}_{C^{\infty}}<\delta\}$, so that their integral $\int_U f\alpha\wedge d\alpha$ is strictly positive or strictly negative. Thus for any $\delta>0$, there exists a $\varepsilon>0$ such that 
\begin{equation}\label{eq:helmod}
    [\mathcal{H}(X)-\varepsilon, \mathcal{H}(X)+\varepsilon]\subset \mathcal{H}\circ P_X\left(\left\{f\, :\, \norm{f}_{C^\infty}<\delta\right\} \right).
\end{equation}
    In other words, by small compactly supported perturbations of $X$ of size $\delta$ in the $C^k$-norm, we can either increase or decrease the helicity of $X$ by any desired amount in an interval $(0, \varepsilon(\delta))$. 

The following lemma shows this can be done uniformly on a $C^1$-neighborhood of $X$. 

\begin{lemma}\label{lem:helcorrec}
    Fix $X\in \mathfrak{X}_\mu^0(M)$ and $U$ an open set where $X$ does not vanish. Given some $\delta'>0$, the following holds: for any vector field $X'$ with $\norm{X-X'}_{C^1}<\delta'$ and for any $\delta>0$, there is a $\varepsilon>0$ (depending on $\delta$ and $\delta'$ but not on $X'$) such that we have
\begin{equation}\label{eq:helcorrec}
    [\mathcal{H}(X')-\varepsilon, \mathcal{H}(X')+\varepsilon]\subset \mathcal{H}\circ P_{X'}\left(\left\{f\, :\, \norm{f}_{C^\infty}<\delta\right\} \right),
\end{equation}
for an appropriate choice of primitive of $\iota_{X'}\mu$ used to define $P_{X'}$.
\end{lemma}
\begin{proof}
By the previous discussion, assume that $$[\mathcal{H}(X)-2\varepsilon, \mathcal{H}(X)+2\varepsilon] \subset \mathcal{H}\circ P_{X}\left(\left\{f\, :\, \norm{f}_{C^\infty}<\delta\right\} \right),$$ and fix some function $f$ with $\norm{f}_{C^\infty}\leq \delta$ such that $\mathcal{H}(P_{X}(f))=\mathcal{H}(X)+2\varepsilon$. Notice that for any $\lambda \in \RR$, $\mathcal{H}(P_{X}(\lambda f))=\mathcal{H}(X)+2\lambda \varepsilon$.

Let $d \alpha'=i_{X'} \mu$. Since $X'$ is $C^{0}$-close to $X$, standard elliptic estimates for the operator $d^{-1}$ (see e.g. Lemma 5 in \cite{CTdL23}) imply that there is a primitive $\alpha'$ for which we have $||\alpha-\alpha'||_{C^{0, s}} \leq C \delta '$ and thus $|\alpha ' \wedge d \alpha' -\alpha \wedge d \alpha| \leq C (\delta')^2$. Define $P_{X'}(f)$ with this primitive. Therefore, we have
\[\cH(P_{X'}(f))=\cH(X')+\int_{\mathcal{U}} f \alpha' \wedge d \alpha' \geq \cH(X')+2\varepsilon-C (\delta')^2. \]

Choose $\delta'$ small enough so that $C \delta'^2<\varepsilon$. Then for any $t \in [-\varepsilon, \varepsilon]$ we can find $\lambda \in [-1, 1]$ such that
\[
\mathcal{H}(P_{X'}(\lambda f))=\mathcal{H}(X')+\lambda \int_{\mathcal{U}} f \alpha' \wedge d \alpha'=\mathcal{H}(X')+t.
\]
 \end{proof}
Let us deduce from this a corollary that we can directly refer to later on.

\begin{cor}\label{cor:localmod}
Let $X$ be a vector field with helicity $h$, and fix an open set $V \subset M$. For any $\delta_1$ small enough, there is a $\delta_2>0$ (going to zero with $\delta_1)$ such that the following holds. Given any exact vector field $Y$ satisfying $\norm{Y-X}_{C^\infty}<\delta_1$, there is a vector field $Z$ with helicity $h$, such that $Z|_{M\setminus V}=Y$ and $\norm{X-Z}_{C^\infty}\leq \delta_2$.
\end{cor}

\begin{proof}
By Lemma \ref{lem:helcorrec} in the open set $V$, given some $\delta, \delta'>0$, there is a $\varepsilon$ such that \eqref{eq:helcorrec} holds. Given $\delta_1$ small enough, there is some $\delta_2(\delta_1)$ going to zero with $\tau$ such that 
$$\norm{P_{X'}(f)-X}<\delta_2$$
for any $X'$ such that $\norm{X-X'}<\delta_1$ and any $f$ such that $\norm{f}<\delta$. Given $Y$ as in the statement, for $\delta_1$ small enough, we can assume that
$$|\mathcal{H}(Y)-\mathcal{H}(X)|<\varepsilon.$$
We can then choose $f\in C^\infty(V)$ satisfying $\norm{f}_{C^\infty}<\delta$ and such that
$$\mathcal{H}(P_{Y}(f))=\mathcal{H}(Y) + (H(X)-H(Y)).$$
The vector field $Z=P_{Y}(f)$ satisfies the required properties.
\end{proof}

\begin{remark}\label{rem:localparam}
The same corollary holds parametrically, i.e. for a parametric perturbation $Y_c$ of $X$ in $U$, we can find a parametric family $X_c$ of zero helicity fields with $X_c|_U=Y_c|_U$.
\end{remark}

\subsection{Kupka-Smale flows}

In this section, we explain why the Kupka-Smale theorem holds for vector fields of fixed helicity. For $h\neq 0$, this can be argued directly using the Kupka-Smale theorem for exact fields, as then one can rescale by a constant to correct helicity. However, an adaptation of the proof of the genericity of Kupka-Smale vector fields is needed for the special case $h=0$. Let us first recall the Kupka-Smale condition.
\begin{definition}
    A volume preserving vector field $X\in \mathfrak{X}_\mu(M)$ satisfies the \textit{Kupka-Smale condition} if it satisfies:
    \begin{enumerate}
        \item The only zeroes of $X$ are hyperbolic,
        \item every periodic orbit is non-degenerate, i.e., the linearized Poincaré first-return map has no root of the unity among its eigenvalues,
        \item given $x_1$ a zero or a hyperbolic periodic orbit of $X$, and $x_2$ another (possibly equal to $x_1$) zero or hyperbolic periodic orbit, the unstable manifold of $x_1$ intersects transversely the stable manifold of $x_2$. 
    \end{enumerate}
\end{definition}
We claim that the following holds.

\begin{thm}\label{thm:kupkah}
    There is a residual subset of vector fields in $\mathfrak{X}_\mu(h)$ satisfying the Kupka-Smale condition, where we endow $\mathfrak{X}_\mu(h)$ with the $C^\infty$-topology. The same holds in $\mathfrak{X}_{\mu,k}(M)$ with the $C^k$-topology.
\end{thm}

First, we claim that the condition (i) above is not only open (this is direct), but also dense in $\mathfrak{X}_\mu(h)$. This is proven in Section \ref{ss:proofmain2}. Next, arguing as in \cite[p 576]{robinson1970generic} for volume preserving vector fields, one easily sees that $\mathfrak{X}_\mu(h)$ with the $C^\infty$-topology is a Baire space. In \cite{robinson1970generic}, an analog of Theorem \ref{thm:kupkah} is proven for several classes of diffeomorphisms and vector fields. The statement and proof for 3D volume-preserving vectors are "folklore". They can be inferred from a combination of the case of volume-preserving vector fields in dimensions $\geq 4$ \cite[Section VIII in p 582]{robinson1970generic} and the case of 1-parametric families of symplectic diffeomorphisms \cite[Section IX in p 586]{robinson1970generic}. In each of these cases, the two kinds of perturbations that one needs within a class of systems for which we want to prove a Kupka-Smale theorem are called "claim $a$" and "claim $b$" in \cite{robinson1970generic}. Their interpretation is morally the following: claim $a$ establishes that one can perturb transversely the flow in an arbitrary fashion, and claim $b$ establishes that the linearization of the first-return map along a periodic orbit can be deformed infinitesimally in an arbitrary fashion. An analog of these claims with fixed helicity is the only requirement for the proof to adapt. We justify these two lemmas by fixing helicity and leave the reconstruction of the argument to the interested reader.\\

Given a vector field $X$, we denote its flow by $\phi_X^t$.
\begin{lemma}["Claim $a$"]\label{lem:kupka1}
   Let $X$ be an exact vector field with helicity $h$. Fix a point $p\in M$ where $X$ does not vanish, and $U$ a flow-box neighborhood of $p$ with coordinates $(x,y,z)$ such that $X=\pp{}{z}$ and $\mu=dx\wedge dy \wedge dz$. Then, for any pair $(a,b)$, there exists a one-parametric family of fields $X_c$ with helicity $h$, with $c\in (-\varepsilon,\varepsilon)$ such that 
   $$\pp{\phi_{X_c}^t(0,0,0)}{c}=(a,b,0).$$
\end{lemma}
\begin{proof}
It is a consequence of the construction described in the proof of \cite[claim a p. 583]{robinson1970generic}, which shows that such a perturbation can be done parametrically among volume-preserving vector fields, combined with a local correction of helicity away from that neighborhood using Corollary \ref{cor:localmod} and Remark \ref{rem:localparam}.
\end{proof}

\begin{lemma}["Claim $b$"]\label{lem:kupka2}
    Let $X$ be an exact vector field with helicity $h$. Let $p\in M$ be a point belonging to a closed orbit $\gamma$ of $X$. Fix a local Poincaré section $\Sigma$ of the orbit through $p$, and let $A\in SL(2,\mathbb{R})$ be the linearization of the first return map at $p$. For any $B\in \mathfrak{sl}(2,\mathbb{R})$, there is a one-parametric family of vector fields $X_c$ with helicity $h$, all having $\gamma$ as a period orbit, such that the linearization of the Poincaré map of $X_c$ at $p$ is $A\cdot \exp(cB)$.
\end{lemma}
\begin{proof}
The construction described in \cite[Claim b p.586]{robinson1970generic} shows how to perturb infinitesimally the linearization of a symplectic diffeomorphism at a fixed point in an arbitrary fashion. A perturbation of the Poincaré first-return map of a periodic orbit can be realized by perturbing the vector field among volume-preserving vector fields near $p$. This can be done, for instance, as in the proof of Lemma \ref{lem:ellipticlinear} below, namely, by introducing a vector field on a flow-box that integrates to an isotopy generating the perturbation near the origin. Finally, one can parametrically correct helicity away from this neighborhood via an application of Corollary \ref{cor:localmod} and Remark \ref{rem:localparam}.
\end{proof}
\subsection{The $C^1$-closing lemma}
We now show that the $C^1$-closing lemma and its consequences hold for vector fields with fixed helicity. As in the Kupka-Smale case, an adaptation of the proof is required when $h=0$.

\begin{thm}[$C^1$-closing lemma with fixed helicity]\label{thm:closingh}
The following holds:
\begin{itemize}
    \item[-] Given $X\in \mathfrak{X}_\mu^0(h)$, there is some $Y \in \mathfrak{X}_\mu^0(h)$ arbitrarily $C^1$-close to $X$ with at least one periodic orbit.
    \item[-] There is a residual set of vector fields in $\mathfrak{X}_{\mu,1}(h)$ satisfying that periodic orbits are dense in $M$.
\end{itemize}
\end{thm}

As explained in \cite{PRclosing} in the Hamiltonian case, Pugh-Robinson's proof of the $C^{1}$-closing lemma adapts to different classes of vector fields if one can establish a property called the "lift axiom" for that class of vector fields. This is done, for example, in \cite[Lemma C.2]{CDHR2} for Reeb flows. We prove an analogous statement for fields with fixed helicity.

\begin{lemma}[Lift axiom for fields with fixed helicity]\label{lem:liftaxiom}
Let $X$ be an exact vector field in $M$ with respect to a volume form $\mu$. Fix some flowbox neighborhood of a point $p$ where $X$ does not vanish, i.e., a neighborhood 
$$U\cong [-1,1]^2\times [-\delta, \delta],$$ 
with coordinates $(x,y,z)$ where the volume form writes $\mu=dx\wedge dy \wedge dz$ and the vector field $X=\pp{}{z}$. There exists $K>0$ and $\varepsilon_0\in(0,1/2)$ such that for every $\varepsilon\in (0,\varepsilon_0)$ and $z_0\in B_\varepsilon(0)\setminus \{0\}\subset [-1,1]^2$, there exists an exact vector field $Y$ with the following properties:
\begin{enumerate}
    \item $\norm{Y-X}_{C^1}< K\varepsilon$,
    \item $Y=X$ in $M\setminus \left(B_{\varepsilon^{-1}|z_0|}(0)\times (-\delta,\delta)\right)$,
    \item the flow of $Y$ takes $(0,0,-\delta)$ to $(z_0,\delta)$,
    \item $\mathcal{H}(Y)=\mathcal{H}(X)$.
\end{enumerate}
\end{lemma}
\begin{proof}
Up to a change of coordinates, we can assume that $z_0=(x_0,0)$. Take a primitive $\alpha$ of $\iota_X\mu$ that is equal to $xdy$ in $U$. Consider the one form
$$\lambda= f(x,y,z)Ax_0ydz,$$
where $f(x,y,z)=a(x)b(y)c(z)$ is given by the product of three even bump functions $a,b,c$ satisfying 
\begin{itemize}
    \item[-] $a(x)=1$ for $x\in (-2|z_0|,2|z_0|)$, and $a(x)=0$ for $|x|\geq \varepsilon^{-1}|z_0|$,
    \item[-] $b(y)=1$ for $y\in (-2|z_0|,2|z_0|)$, and $b(y)=0$ for $|y|\geq \varepsilon^{-1}|z_0|$,
    \item[-] $c(z)=1$ for $z\in (-\delta/2,\delta/2)$, and $c(z)=0$ near $z=-\delta$ and $z=\delta$.
\end{itemize}

In particular, the support of $\lambda$ is contained in $B_{\varepsilon^{-1}|z_0|}(0)\times (-\delta,\delta)$. The function $f$ can be chosen so that $\norm{\nabla f}_{{C^{0}}}=O(\varepsilon|z_0|^{-1})$, and $\norm{\nabla^2 f}_{{C^{0}}}=O(\varepsilon^2|z_0|^{-2})$. Let $Y$ be the vector field satisfying $\iota_Y\mu=d(\alpha+\lambda)$, which satisfies $Y=X$ in $M\setminus \left(B_{\varepsilon^{-1}|z_0|}(0)\times (-\delta,\delta)\right)$. Along $B_{2|z_0|}(0)\times (-\delta,\delta)$, it has the expression
$$Y=X+c(z)x_0A\pp{}{x}.$$
It is then clear that if we choose an adequate constant $A$ (and this only depends on the function $c(z)$, which is independent of everything else), the flow of $Y$ takes $(0,0,-\delta)$ to $(z_0,\delta)$. To prove that $X$ and $Y$ are $K\varepsilon$-close in the $C^1$ topology, we need to show that the $C^2$-norm of $\lambda$ is smaller than $K\varepsilon$ for some universal constant $K$. We have 
\begin{align*}
    \norm{Ax_0y}_{C^0}&<A|x_0|\delta,\\ 
\norm{\nabla(Ax_0y)}_{C^0}&<A|x_0|,\\
\norm{\nabla^2(Ax_0y)}_{C^0}&=0.
\end{align*}
Introducing the notation $G=fAx_0y$, we have the bound
\begin{align*}
    \norm{fAx_0y}_{C^0}&\leq \norm{f}\norm{Ax_0y} \\
        &< A|x_0|\frac{|x_0|}{\varepsilon}\\
        &< A\varepsilon,
\end{align*}
the bound
\begin{align*}
    \norm{\nabla(fAx_0y)}_{C^0} &\leq \norm{\nabla(f)}_{C^0}\norm{Ax_0y} + \norm{\nabla (Ax_0y)}\norm{f}\\
    &< \varepsilon|x_0|^{-1}A|x_0|\frac{|x_0|}{\varepsilon} + A\varepsilon\\
    &<2A\varepsilon,
\end{align*}
and the bound
\begin{align*}
    \norm{\nabla^2(fAx_0y)}&\leq \norm{\nabla^2f}\norm{Ax_0y}+2\norm{\nabla (f)}\norm{\nabla(Ax_0y)} + \norm{f}\norm{\nabla^2{Ax_0y}}\\
    &<\varepsilon^2|x_0|^{-2}A|x_0|\frac{|x_0|}{\varepsilon} + 2\varepsilon|x_0|^{-1}A|x_0|\\
    &<3\varepsilon A.
\end{align*}
This shows that $Y$ satisfies 
$$\norm{Y-X}_{C^1}< K\varepsilon, \qquad \text{with} \qquad K=3A,$$
and in particular, $K$ does not depend on anything. Finally, let us check that the helicity of $Y$ is the same as that of $X$. The one-form $\beta=xdy+\lambda$ is the one such that $\iota_Y\mu=d\beta$. Then we have

\begin{align*}
    \int_U \beta\wedge d\beta&=\frac{2x_0}{\delta}\int_U \left(-xy\pp{f}{x}+fy\right)\mu\\
                            &=\frac{2x_0}{\delta} \int_{-1}^1 \left(\int_{-1}^1\left( \int_{-\delta}^\delta \left((-xa'+a)bcy\right) \,dx\right)\, dy \right)\, dz\\
                            &=\frac{2x_0}{\delta}\int_{-1}^1c \left(\int_{-1}^1 (-xa'+a) \left( \int_{-\delta}^\delta by \,dy\right)\, dx \right)\, dz,
\end{align*}
and since $b$ is an even function, we deduce that $\int_{-\delta}^\delta by \,dy=0$ and hence that $\int_U \beta\wedge d\beta=0$. Since $\beta$ extends, while being a global primitive of $\iota_Y\mu$, as $\alpha$ away from $U$, we conclude that $\mathcal{H}(Y)=\mathcal{H}(X)$.
\end{proof}

\subsection{Elliptic orbits and a Franks-type lemma}\label{ss:hellipticorbits}
The goal of this section is to correct an issue in the proof of \cite[Theorem 1]{BeDu} and to show that its statement can be adapted to vector fields of fixed helicity.
\begin{thm}\label{thm:fixedhelliptic}
Let $M$ be a closed three-manifold.
\begin{itemize}
    \item[-] Let $X$ be a non-Anosov exact vector field with helicity $h$. Given $U\subset M$, there exists a smooth exact vector field $Y$ with helicity $h$ that admits an elliptic periodic orbit intersecting $U$ and that is arbitrarily $C^1$-close to $X$.
    \item[-] In the set of far-from-Anosov\footnote{A vector field is far-from-Anosov in $\mathfrak{X}_{\mu,1}^0$ if it admits a $C^1$-neighborhood in $\mathfrak{X}_{\mu,1}^0$ without Anosov vector fields.} vector fields in $\mathfrak{X}_{\mu,1}^0$, there is a residual set of vector fields whose set of elliptic periodic orbits is dense.
\end{itemize}  
\end{thm}
In our proof of Theorem \ref{thm:main}, we will only need the first item above. We point out that the first item, for $h\neq 0$, follows directly by applying Besa-Duarte's result and correcting by a constant rescaling. However, their result still needs the correction that we will describe here. Indeed, the proof of \cite[Theorem 1]{BeDu} relies on a "pasting lemma" \cite[Lemma 3.2]{BeDu}, which is taken from \cite{AM}, but which is wrong. The lemma was later corrected in \cite{Tei}, but that version of the lemma does not work in the proof of \cite[Theorem 1]{BeDu}. The Lemma is applied in \cite[Lemma 3.7]{BeDu} and \cite[Lemma 3.12]{BeDu}, which then imply the key propositions \cite[Proposition 3.8]{BeDu} and \cite[Proposition 3.13]{BeDu}. An inspection of the proof makes it clear that what is needed is to be able to modify the linearization of the vector field along pieces of periodic orbits by a small amount $\varepsilon$, by a compactly supported perturbation of the vector field of  $C^{1}$-size $\delta(\varepsilon)$, where $\delta$ is uniform in $M$ and does not depend on which piece of orbit we are choosing, its length, or the thickness of the neighborhood where we are allowed to perturb. This kind of perturbations is exactly provided by what is sometimes called Franks' lemma (for pieces of orbit, as there is one for points), which is proven in \cite[Theorem A.2]{BGV} for non-conservative vector fields. A Hamiltonian version of it is usually attributed to a preprint of Vivier (see \cite[Theorem 5.2]{BRT}), which is no longer accessible.\footnote{A weaker statement is proved in \cite{AD} for Poisson manifolds, but in that version the $\delta$ in the perturbation is not uniform among all pieces of orbit, which is a key property needed here.} For 3D volume-preserving vector fields, a proof of Franks' lemma for pieces of orbit is lacking in the literature, and only one is for points and is proven in \cite{Tei}. For this reason, we establish a Franks lemma for pieces of orbit of 3D volume-preserving fields, which is enough to correct the proof of \cite[Theorem 1]{BeDu}. We then explain how the proof of \cite[Theorem 1.2]{BeDu} can be adapted to fix helicity.
\medskip

\paragraph{\textbf{A Franks' lemma for pieces of orbit.}} 
 
 Let $X$ be a smooth $\mu$-preserving vector field on $M$. Fix a point $p \in M$ and let $\Gamma$ be the piece of orbit $\Gamma=\{\phi^{t}_{X}(p), \,\,\, t \in [0, T]\}$. 

Fix a trivialization $\tau$ of the normal bundle of $X$ for every point in $\Gamma$:
\[
\tau: T \Gamma/ \text{span } X \to \RR^2,
\]
that we denote fiberwise as 
\[
\tau_p: T_p M \to \RR^2.
\]
In the case that $X$ has no zeroes and has trivial Euler class, $\tau$ can be chosen to be a global trivialization of the normal bundle, but this won't be needed in what follows. 

The trivialization above defines a map from $[0, T]$ to $SL(2, \RR)$:
\[
A_{X}: [0, T] \to SL(2, \RR) \qquad\text{given by}\qquad A_{X}(t)=\tau_{\phi^{t}_{X}(p)} \circ D_{p} \phi^{t}_{X} \circ \tau^{-1}_{p},
\]
as well as the associated map 
\[
\mathfrak{a}_{X}: [0, T] \to \mathfrak{sl}(2, \RR) \qquad\text{given by}\qquad \mathfrak{a}_{X}(t)=\bigg(\frac{d}{dt} A_{X}(t)\bigg) \circ A_{X}^{-1}(t).
\]
In other words, $A_{X}(t)$ is a path in $SL(2, \RR)$ whose velocity is $\mathfrak{a}_{X}(t)\circ A_{X}(t) \in T_{A_{X}(t)} SL(2, \RR)$.

\begin{remark}
For convenience, in what follows we fix an arbitrary Riemann metric $g$ on $M$ whose induced volume is $\mu$: all $C^k$-norms of vector fields and other sections on $M$ will be understood with respect to this metric. The norms of matrices in $\mathfrak{sl}(2, \RR)$ will be taken to be the supremum of the matrix elements.    
\end{remark}

Let $v_{1}(t):=\tau_{\phi^{t}_{X}(p)}^{-1}(\frac{\partial}{\partial x}) \in T_{\phi^{t}_{X}(p)}M$ and $v_{2}(t):=\tau_{\phi^{t}_{X}(p)}^{-1}(\frac{\partial}{\partial y})\in T_{\phi^{t}_{X}(p)}M$. Define $C_{\tau}$ to be a constant satisfying
\[
\frac{1}{C_{\tau}} \leq |v_{i}(t)| \leq C_{\tau}.
\]

\begin{remark}
We can always rescale any given trivialization so that $C_{\tau}=1$, but this won't be assumed in what follows. 
\end{remark}

The key technical lemma to correct the main theorem in \cite{BeDu} is the following. To avoid interrupting the flow of the exposition, the proof is deferred to Appendix \ref{app}.

\begin{lemma}[Frank's-type Lemma for volume-preserving flows]\label{lem:franks}
Fix any $\varepsilon>0$ and a $C^1$ map $\mathfrak{c}: [0, T] \to \mathfrak{sl}(2, \RR)$, compactly supported in $[0, T]$, and satisfying
\begin{equation}\label{psize}
\sup_{t \in [0, T]} ||\mathfrak{a}_{X}(t)-\mathfrak{c}(t)|| < \varepsilon.
\end{equation}
There is a constant $K$, depending only on $||X||_{C^{0}}$ and $C_{\tau}$, such that the following holds: given any arbitrarly small tubular neighborhood of $\Gamma$, there is a vector field $Y$ in $\mathfrak{X}^{0}_{\mu}(M)$ with the following properties:
\begin{enumerate}
\item $X=Y$ outside the tubular neighborhood. 
\item $||X-Y||_{C^{1}(M)} \leq K \varepsilon$. 
\item For any $t\in [0, T]$, $\phi^{t}_{X}(p)=\phi^{t}_{Y}(p)$.
\item For any $t\in [0, T]$ we have
\[
\bigg(\frac{d}{dt} A_{Y}(t)\bigg) \circ A_{Y}^{-1}(t)=\mathfrak{c}(t),
\]
where
\[
A_{Y}(t):=\tau_{\phi^{t}_{Y}(p)} \circ D_{p} \phi^{t}_{Y} \circ \tau^{-1}_{p}.
\]

  \end{enumerate}
\end{lemma}

\begin{remark}
As the proof makes manifest, the Lemma can be generalized, mutatis mutandis, for volume-preserving vector fields (not necessarily exact) in any dimension $n$.
\end{remark}

\begin{remark}
It is crucial for the application of Lemma \ref{lem:franks} in the correction of the proof of \cite[Theorem 1]{BeDu} and subsequently in Theorem \ref{thm:fixedhelliptic} that the constant $K$ does not depend on the point $p$ or the orbit $\{\phi^{t}_{X}(p), t \in [0, T]\}$, this is why some additional considerations must be made when expressing the problem locally in flow-box coordinates (see Step 2 in the proof below). Compare this with the constant $K$ in Lemma \ref{lem:liftaxiom}, where $K$ depends on the point where one perturbs the field $X$. 
\end{remark}

\paragraph{\textbf{Sketch of Theorem \ref{thm:fixedhelliptic}.}} We are now ready to sketch the adaptations required to deduce Theorem \ref{thm:fixedhelliptic}.

\begin{proof}[Proof of Theorem \ref{thm:fixedhelliptic}]
Following the proof of \cite[Theorem 1]{BeDu}, we first assume that $X$ has dense periodic orbits, satisfies the Kupka-Smale property, and is not Anosov, by using Theorems \ref{thm:kupkah} and \ref{thm:closingh}. This possibly produces a vector field that has only $C^1$-regularity. The whole argument will adapt as long as we can produce the perturbations given by \cite[Proposition 3.8 and Proposition 3.13]{BeDu}, which produce an elliptic period orbit, but keeping the helicity fixed. To illustrate how to keep helicity fixed while producing perturbations with Lemma \ref{lem:franks}, we detail the case of \cite[Proposition 3.8]{BeDu} (the other one being analogous).

Let $2\delta$ be the allowed $C^{1}$-size of the perturbation by which we want to produce an elliptic periodic orbit. First, Lemma \ref{lem:franks} allows us to perturb by some $\varepsilon$ the linearization along any piece of orbit, in any $\kappa$-neighborhood of that piece of orbit for any $\kappa>0$, by perturbations of the vector field of size $\delta$. The proof of Proposition 3.8 in \cite{BeDu} finds a piece $\Gamma$ of a certain periodic orbit $\gamma_{\delta}$ of $X$ for which a $\varepsilon$-perturbation of its linearization is needed to make it elliptic.  Choose some open set $V_\delta$ in the complement of this particular (non-degenerate) periodic orbit, where we will correct helicity. First, we use Zuppa's regularization theorem\footnote{The regularized vector field in Zuppa's theorem is obtained by doing compactly supported perturbations on local charts. In particular, if the original vector field is exact, so is the regularized one.} \cite{Zu} to make $X$ smooth. This can be done by a perturbation of arbitrarily small size $\eta<<\delta$, so that:
\begin{itemize}
    \item[-] the vector field has a periodic orbit given by a perturbation of $\gamma_{\delta}$, that we still denote by $\gamma_{\delta}$,
    \item[-] a $\delta$-perturbation along the corresponding piece $\Gamma$ of $\gamma_{\delta}$ is still enough to make $\gamma_{\delta}$ elliptic.
\end{itemize}  
Applying Corollary \ref{cor:localmod} allows us to regularize while keeping the helicity of $X$ fixed. However, notice that to perturb the transverse linearization along the piece of $\gamma_{\delta}$, we cannot directly correct helicity using Corollary \ref{cor:localmod}. Indeed, given $\delta$ the size of perturbation, the $\delta_2$ of Corollary \ref{cor:localmod} could be larger than $\delta$. We proceed differently, using Lemma \ref{lem:helcorrec} directly instead, and the fact that the support of the perturbation given by Lemma \ref{lem:franks} can be chosen to be an arbitrarily small neighborhood of the piece of $\gamma_\delta$.

Lemma \ref{lem:helcorrec} tells us a bit more than Corollary \ref{cor:localmod}: for any vector field $\delta$-close to $X$, we can correct its helicity by any value in $[-\tau,\tau]$ for any small enough fixed $\tau>0$ by a perturbation of size $\hat \delta$, where $\hat \delta$ can be chosen arbitrarily small (in particular smaller than $\delta$) if $\tau$ is chosen small enough. Choosing $\kappa$ small enough in the application of Lemma \ref{lem:franks}, we perturb the vector field $X$ to some $Y'$ such that the linearization of $Y'$ along $\gamma_{\delta}$ is the needed one, but keeping the helicity of $Y'$ at a distance smaller than $\tau$ from that of $X$. This is possible because the distance between $X$ and $Y'$ is always $\delta$, while the support of the perturbation can be made of arbitrarily small volume (because of $\kappa$ being arbitrarily small), thus making the helicity of $Y'$ arbitrarily close to that of $X$. This produces a vector field $Y$ with an elliptic periodic orbit, the same helicity as $X$, and such that $\norm{Y-X}_{C^1}< \delta+ \eta<2\delta$.
\end{proof}

\section{Ruelle invariant near non-Anosov vector fields} \label{sec:Ruelle}

In this section, we discuss a version of the Ruelle invariant \cite{ruelle1985rotation} for volume-preserving vector fields with vanishing Euler class. The goal is to prove Theorem \ref{thm:main_ruelle_prop}, which states that the Ruelle invariant is non-constant in any open set of vector fields of given helicity that contains a non-Anosov vector field. 

The strategy of the proof is as follows. We first show that if a vector field admits a toric flow tube (Definition \ref{def:toric_flow_tube}), which is a certain domain diffeomorphic to $I\times T^2$ where the vector field and the volume form admit a nice normal form, then it is possible to $C^\infty$-perturb the vector field so that the Ruelle invariant changes but the helicity does not. This requires computing the Ruelle invariant (with respect to a given trivialization) and the helicity in a toric flow tube. We then apply Theorem \ref{thm:fixedhelliptic} to show that given a non-Anosov vector field, it is possible to $C^1$-perturb it to one that does admit such a toric flow tube, while keeping the same helicity. 

\subsection{The Ruelle invariant} \label{subsec:Ruelle_construction} We begin with a review of the construction of the Ruelle invariant following the general treatment of \cite{ce2022} (also see \cite{ce2021}). In the following section, fix
\[
\text{a compact $3$-manifold $M$ with boundary} \qquad\text{and}\qquad \text{a volume form $\mu$.}
\]
Also fix a nowhere vanishing $C^1$ vector-field $X$ preserving $\mu$ and generating a flow
\[
\phi:\R \times M \longrightarrow M.
\]
We assume that the Euler class of the vector field (or equivalently, of the normal bundle) is zero. In particular, we may choose a trivialization
\[
\tau:TM/\on{span}(X) \xrightarrow{\sim} M \times \R^2,
\]
denoted fiberwise by $\tau_x:T_xM \longrightarrow \R^2$.
\begin{definition}[Ruelle density] \label{def:Ruelle_density} The \emph{Ruelle density} of $V$ and $\tau$ is the $L^1$-function
\[\ru(X,\tau) \in L^1(M,\mu)
\]
given by the following construction. Consider the map from $\R \times M$ to the special linear group.
\begin{equation} \label{eq:linearized_flow_in_tau}
\Phi:\R \times M \longrightarrow \SL(2) \qquad\text{given by}\qquad \Phi(t,x) = \tau_{\phi(t,x)} \circ D\phi_{t, x} \circ \tau_x^{-1}.
\end{equation}
The map $\Phi$ lifts uniquely to a map to the universal cover of the special linear group
\[
\widetilde{\Phi}:\R \times M \longrightarrow \twSL(2) \qquad\text{with}\qquad \widetilde{\Phi}|_{0 \times M} = \widetilde{\on{Id}} .
\]
The Ruelle density is defined to be the limit 
\begin{equation} \label{eq:ru_limit}
\ru(X,\tau) := \lim_{T \longrightarrow \infty} \frac{1}{T}\rho \circ \widetilde{\Phi}_T,\end{equation}
in $L^1(M,\mu)$.
Here $\rho:\twSL(2) = \twSp(2) \longrightarrow \R$ is any choice of rotation quasimorphism \cite[Def 2.5]{ce2022}. 
\end{definition} 
An example of such a quasimorphism is the standard rotation number, which we describe following \cite{h2019}. We represent the lift of an element $A\in SL(2)$ to $\widetilde{SL}(2)$ as a path $\tilde A=A_t$ such that $A_0=\operatorname{Id}$ and $A_1=A$. Given a non-zero vector $v$ in $\mathbb{R}^2$, let $\theta(v)$ be the value such that the angular coordinate of $v_t=A_tv$ lifts to a function $2\pi\theta(v)$ with values in $\mathbb{R}$. Then the rotation number $\tilde A$ is defined as $$\rho(\tilde A)= \lim_{n\rightarrow \infty} \frac{1}{n} \sum_{k=1}^n \theta(A^{k-1}v),$$ 
and this value does not depend on the initial choice of $v$.

\begin{remark}
    There are alternative descriptions of the Ruelle density; one of them is as follows. Assuming that $M$ is a homology sphere, consider the QR factorization of $\Phi$, which defines a map $u: \RR \times M \rightarrow SO(2) \cong \mathbb{S}^{1}$. For every $t \in \RR$, consider the closed (and hence exact) one form $\alpha_t=u_{t}^{*} d \theta$, and choose a primitive $d \rho_{t}=\alpha_t$ that varies smoothly in $t$. Then $\text{ru}(X, \tau)=\lim_{T \rightarrow \infty} \rho_{T}/T$.
\end{remark} 

\begin{lemma} The limit (\ref{eq:ru_limit}) exists in $L^1(M,\mu)$ and depends only on the isotopy class $[\tau]$ of $\tau$.  
\end{lemma}

\begin{proof} By \cite[Prop 2.13]{ce2021}, the limit exists, and only depends on the isotopy class of $\tau$. Note that \cite[Prop 2.13]{ce2021} is stated for Reeb flows, but the proof is the same for volume preserving flows. Alternatively, this is a case of \cite[Prop 3.11 and Prop 3.13(d)]{ce2022} applied to the $\phi$-cocycle
\[(E,\Phi) \qquad\text{where}\qquad E = TM/\on{span}(X) \quad\text{and}\quad \Phi = D\phi\qedhere\]\end{proof}

\begin{definition} The \emph{Ruelle invariant} $\Ru(X,\tau)$ of the vector-field $X$ and trivialization $\tau$ is given by
\[
\Ru(X,\tau) := \int_M \ru(X,\tau) \cdot \mu.
\]
\end{definition}

We will require a few basic properties of the Ruelle invariant. First, the Ruelle invariant satisfies the following additivity property under disjoint union, which follows easily from the definition.

\begin{prop}[Disjoint union] \label{prop:Ru_disjoint} If $M = A \cup B$ where $A$ and $B$ are codimension zero, $\phi$-invariant compact sub-manifolds with $\partial A = \partial B$, then
\[
\on{Ru}(X,\tau) = \on{Ru}(X|_A,\tau|_A) + \on{Ru}(X|_B,\tau|_B).
\]
\end{prop}

Second, the Ruelle invariant is continuous in the $C^1$-topology on vector fields. We will state a relatively simple version of this property. Assume that $M$ satisfies
\[
H^1(M;\Z) = 0 \qquad\text{and}\qquad H^2(M;\Z) = 0.
\]
In this case, the Ruelle invariant is well-defined for any non-vanishing $C^1$ vector field $V$ preserving $\mu$, and for any vector field there is a unique choice of isotopy class of trivialization of $TM/\on{span}(X)$. It thus determines a well-defined map
\[
\Ru:\mathfrak{X}_{\mu, 1}(M) \longrightarrow \R,
\]
where $\mathfrak{X}_{\mu, 1}(M)$ stands for the space of $C^{1}$ volume-preserving vector fields.
\begin{prop}[Continuity] \label{prop:Ru_continuity} The Ruelle invariant $\Ru$ is a continuous function on $\mathfrak{X}_{\mu, 1}(M)$ with respect to the $C^1$-topology, assuming that $H^1(M;\Z) = H^2(M;\Z) = 0$.
\end{prop}

\begin{proof} The proof is essentially identical to \cite[Prop 2.13(c)]{ce2021} which is stated for Reeb flows. 

\vspace{3pt}

{\bf Step 1: Preliminaries.} We require some preliminary setup and notation before proceeding to the main proof. As in Definition \ref{def:Ruelle_density}, fix a rotation quasimorphism
\[{\rho}:\twSp(2) \longrightarrow \R.\]
Note that by the quasimorphism property \cite[Def 2.3]{ce2022}, there is a constant $C > 0$ such that
\begin{equation} \label{eq:quasimorphism_property}
|\rho(\widetilde{A}\widetilde{B}) - \rho(\widetilde{A}) - \rho(\widetilde{B})| < C \qquad\text{for any pair of elements }\widetilde{A},\widetilde{B} \in \twSp(2)
\end{equation}
Next, fix a vector field $X \in \mathfrak{X}_{\mu, 1}(M)$. Choose a transverse $2$-plane field $E \subset TM$ and a trivialization $\tau:E \longrightarrow \R^2$. Note that there is a neighborhood $\mathcal{U}$ of $X$ in $\mathfrak{X}_{\mu, 1}(M)$ such that any $W \in \mathcal{U}$ is also transverse to $E$. In particular, there is a natural isomorphism
\[
\sigma_W:E \xrightarrow{\sim} TM/\on{span}(W) \qquad\text{for any }W \in \mathcal{U}
\]
given by composing the inclusion $E \subset TM$ with the projection $TM \longrightarrow TM/\on{span}(W)$. This yields a trivialization $\tau_W = \tau \circ \sigma_W^{-1}$ of $TM/\on{span}(W)$ for any $W \in \mathcal{U}$. We let
\[
\Phi_W:\R \times M \longrightarrow \Sp(2)
\]
denote the map (\ref{eq:linearized_flow_in_tau}) with respect to the flow of $W$ and the trivialization $\tau_W$. Likewise, we let
\[
\widetilde{\Phi}_W:\R \times M \longrightarrow \twSp(2)
\]
denote the corresponding lift of $\Phi_W$ to $\twSp(2)$. Note that both maps are continuous in $W$ in the compact open topology. If $\phi_t^W$ denotes the flow of $W$, both maps satisfy the cocycle properties
\[
\Phi_W(s+t,x) = \Phi_W(s,\phi_t^W(x))\Phi_W(t,x),\]
and
\[ \widetilde{\Phi}_W(s+t,x) = \widetilde{\Phi}_W(s,\phi_t^W(x))\widetilde{\Phi}_W(t,x).\]
Finally, we let $\ru_{W,T}: M \rightarrow \R$ denote the function
\[
\qquad \ru_{W,T}(x) = \frac{\rho \circ \widetilde{\Phi}_W(T,x)}{T},
\]
associated to $W$.

\vspace{3pt}

{\bf Step 2: Approximation.} Next, we prove the following approximation formula
\begin{equation} \label{eq:T_approximation}
\Big|\Ru(W) - \int_M \ru_{W,T} \cdot \mu \Big| \le \frac{C \cdot \vol{M,\mu}}{T}
\end{equation}
for the Ruelle invariant. To derive this formula, we first note that by the quasimorphism property (\ref{eq:quasimorphism_property}) we have
\[
|{\rho} \circ \widetilde{\Phi}_W(nT,x) - \sum_{k=0}^{n-1} {\rho} \circ \widetilde{\Phi}_W(T,\phi^k_T(x))| \le Cn \qquad\text{for any $n,T$ > 0}
\]
By rewriting this estimate in terms of $\ru_{W,T}$ and dividing by $nT$, we see that
\[
|\ru_{W,nT} - \frac{1}{n} \sum_{k=0}^{n-1} \ru_{W,T} \circ \phi^k_T| \le \frac{C}{T} \qquad\text{for any $n,T$ > 0}
\]
We then integrate over $M$ and take the limit as $n \longrightarrow \infty$ to acquire (\ref{eq:T_approximation}) as follows. 
\[
\Big|\Ru(W) - \int_M \ru_{W,T} \mu \Big| = \lim_{n \longrightarrow \infty} \Big|\int_M (\ru_{W,nT} - \ru_{W,T})\mu\Big|
\]
\[
= \lim_{n \longrightarrow \infty} \Big|\int_M (\ru_{W,nT} - \frac{1}{n}\sum_{k=0}^{n-1} \ru_{W,T} \circ \phi^k_T)\mu\Big| \le \frac{C \cdot \vol{M,\mu}}{T}
\]
We use the fact that $\phi_T$ is volume preserving to go from the first to the second line above.

\vspace{3pt}

{\bf Step 3: Conclusion.} Finally we move on to the main proof. Choose a vector field $X \in \mathfrak{X}_{\mu, 1}(M)$ and an $\varepsilon > 0$. We want to show that there is a $\delta > 0$ such that
\[
|\Ru(W) - \Ru(X)| < \varepsilon \qquad\text{if}\qquad \|W - X\|_{C^1} < \delta,
\]
where $W$ is another element in $\mathfrak{X}_{\mu, 1}(M)$.
We assume $\delta$ is small enough so that $W$ admits the trivialization $\tau_W$ constructed above. By applying (\ref{eq:T_approximation}), we have 
\[
\Big|\Ru(W) - \int_M \ru_{W,T} \cdot \mu \Big| < \varepsilon/3\qquad\text{and}\qquad \Big|\Ru(X) - \int_M \ru_{X,T} \cdot \mu \Big| < \varepsilon/3,
\]
for $T$ sufficiently large. Moreover, recall that $\widetilde{\Phi}_W$ varies continuously in the compact open topology, with respect to the $C^1$-topology on $\mathfrak{X}_{\mu, 1}(M)$. Since $\widetilde{\Phi}_X([0,T] 
\times M) \subset \twSp(2)$ is compact and $\rho$ is also continuous, this implies that if $\|W - X\|_{C^1}$ is sufficiently small, then
\[
\|\ru_{W,T} - \ru_{X,T}\|_{C^0} \le \sup_{[0,T] \times M} |\rho \circ \widetilde{\Phi}_{W,T}(x) - \rho \circ \widetilde{\Phi}_{X,T}(x)| < \frac{\varepsilon}{3 \vol{M,\mu}}.
\]
By combining the previous three inequalities, we get the desired inequality, namely
\[
|\Ru(W) - \Ru(X)| \le 2\varepsilon/3 + \|\ru_{W,T} - \ru_{X,T}\|_{C^0} \cdot \vol{M,\mu} < \varepsilon. \qedhere
\]\end{proof}
Lastly, the Ruelle functional is invariant along coadjoint orbits.
\begin{lemma}\label{lem:Ruellecoadjoint}
    Let $X\in \mathfrak{X}_\mu^0(M)$ be a non-vanishing exact vector field in an integral homology sphere $M$. Then, we have
    $$\Ru(\varphi_*X)=\Ru(X),\qquad \text{for any} \qquad \varphi \in \SDiff_\mu(M).$$
\end{lemma}
\begin{proof}
Fix a trivialization 
$$\tau : TM/\on{span}(X) \longrightarrow \mathbb{R}^2,$$
and denote by $\hat \tau$ the trivialization of $TM/\on{span}(\varphi_*X)$ given by 
\begin{align*}
    \hat \tau : TM/\on{span}(\varphi_*X) &\longrightarrow \mathbb{R}^2\\
                v &\longmapsto \tau\circ d\varphi^{-1}.
\end{align*}
Finally, we denote by $\psi$ and $\Psi$ the flow and the corresponding map to $SL(2,\mathbb{R})$ associated with the vector field $\varphi^*X$. Then we have
\begin{align*}
     \Phi(t,x)\circ (\operatorname{id}\times\varphi^{-1})&= \tau_{\phi(t,\varphi^{-1}(x))} \circ D\phi_{t, \varphi^{-1}(x)} \circ \tau_{\varphi^{-1}(x)}^{-1}\\
     &= \tau_{\phi(t,\varphi^{-1}(x))} \circ D\varphi^{-1}\circ D\varphi \circ D\phi_{t, \varphi^{-1}(x)} \circ D\varphi^{-1}\circ D\varphi \circ\tau_{\varphi^{-1}(x)}^{-1}\\
     &=\hat \tau_{\psi(t,x)} \circ D\psi_{t,x} \circ \hat \tau^{-1}_x\\
     &=\Psi(t,x),
\end{align*}
and thus the following commutative diagram
\begin{center}
\begin{tikzcd}[column sep=4em, row sep=4em]
  \mathbb{R}\times M \arrow[r,"\Psi"] 
    \arrow[d,"{\mathrm{id}\times \varphi^{-1}}"'] 
  & SL(2) \\
  \mathbb{R}\times M \arrow[r,"\widetilde{\Phi}"'] \arrow[ur,"\Phi"] 
  & \widetilde{SL(2)} \arrow[u,"\pi"]
\end{tikzcd}
\end{center}
from which one easily deduces that $\widetilde \Psi= \widetilde \Phi\circ (\operatorname{id}\times \varphi^{-1})$. Finally, using this equality in the formula of the Ruelle density \eqref{eq:ru_limit}, we can deduce
\begin{align*}
\Ru(X)&= \int_M \ru(X)d\mu\\
        &= \int_M \varphi^*\ru(X)d\varphi^*\mu\\
        &=\int_M \ru(\varphi_*X)d\mu\\
        &=\Ru(\varphi_*X),
\end{align*}
as we wanted to prove.
\end{proof}
In fact, as shown by Gambaudo and Ghys \cite[Théorème 3.5]{GG}, the Ruelle invariant (on integral homology spheres) is preserved even under conjugation by volume-preserving homeomorphisms. The analogous result for the helicity is a well-known open problem, recently proven for non-vanishing vector fields by Edtmair and Seyfaddini \cite{ES05}.

\subsection{Helicity in manifolds with boundary} \label{subsec:helicity} We next discuss the helicity in the case of volume-preserving vector fields on a manifold with boundary. Fix a compact $3$-manifold with boundary
\[
M \qquad\text{with volume form}\qquad \mu.
\]
Let $X$ be an exact volume preserving vector field on $(M,\mu)$ tangent to the boundary. That is
\[
d(\iota_X\mu) = 0, \qquad [\iota_X\mu] = 0 \in H^1(M;\R), \qquad\text{and}\qquad \iota_X\mu|_{\partial M} = 0.
\]

In this setting, there is a certain homological ambiguity in the construction of the helicity. To be precise, suppose that $\alpha \in \Omega^1(M)$ is a primitive of $\iota_X\mu$, so that $d\alpha = \iota_X\mu$. Then we have
\[
d(\alpha|_{\partial M}) = (d\alpha)|_{\partial M} = \iota_X\mu|_{\partial M} = 0.
\]
Thus there is an associated class $[\alpha|_{\partial M}] \in H^1(\partial M;\R)$. Given two primitives $\alpha$ and $\beta$, we have $d(\beta - \alpha) = 0$ and thus
\[
 [\jmath^*\alpha] - [\jmath^*\beta] \in \jmath^*H^1(M;\R)
\]
where $\jmath: \partial M \hookrightarrow M$ is the inclusion map. Thus we have a well-defined cohomological invariant of $(X,\mu)$ as follows.
\begin{definition} The \emph{set of helicity classes} $S(X,\mu)$ is the equivalence class of elements
\[
S(X,\mu) = \big\{[\alpha|_{\partial M}] \; : \; d\alpha = \iota_X\mu\big\} \subset H^1(\partial M;\R).
\]
Note that this admits a free and transitive action by the image of $H^1(M;\R)$ in $H^1(\partial M;\R)$. \end{definition}

\noindent We may unambiguously define the helicity after fixing a helicity class.

\begin{definition} The \emph{helicity} $\mathcal{H}(X,\theta) \in \R$ of an exact volume preserving vector field $X$ on $(M,\mu)$ and a helicity class $\theta \in S(X,\mu)$ is given by
\[
\mathcal{H}(X,\theta) = \int_M \alpha \wedge \iota_X\mu, \qquad\text{where}\qquad [\alpha|_{\partial M}] = \theta.
\] \end{definition}

\begin{lemma} The helicity $\mathcal{H}(X,\theta)$ is well-defined (i.e. independent of primitive $\alpha$ in the helicity class $\theta$).
\end{lemma}

\begin{proof} If $[\alpha|_{\partial M}] = [\beta|_{\partial M}] = \theta$, then the corresponding integrals satisfy
\begin{align*}
    \int_M \alpha \wedge \iota_X \mu -  \int_M \beta \wedge \iota_X\mu &= -\int_M d((\alpha - \beta) \wedge \alpha) \\
    &= -\int_{\partial M} (\alpha - \beta) \wedge \alpha\\
    &=-\int_{\partial M} df\wedge \alpha\\
    &=-\int_{\partial M} d(f\alpha)\\
    &=0,
\end{align*}
where $f\in C^\infty(\partial M)$ is a primitive of $\alpha-\beta$ restricted to $\partial M$, and in the last two equalities we used that $d\alpha$ vanishes along $\partial M$ and Stokes theorem respectively.
\end{proof}

%\begin{remark}
%The set of helicity classes $S(X, \mu)$ is, in general, just a proper subspace of $H^{1}(\partial M, \RR)$. For example, when $M$ is a solid torus $M=\mathbb{D} \times \mathbb{S}^{1}$, with coordinates $(r, \theta) \in \mathbb{D}$ and $\varphi \in \mathbb{S}^{1}$, the set of helicity classes is 1-dimensional and generated by $d \varphi$, as the 1-dimensional subspace of closed 1-forms on $\partial M$ generated by $d \theta$ cannot be extended to the interior.

%\end{remark}

\begin{remark}
The quantity $\cH(X, \theta)$ is related to, but not the same as, the \emph{relative helicity} used in magnetohydrodynamics, see e.g. \cite[Section 1]{al1991}. In an euclidean domain $M \subset \RR^3$, $\cH(X, \theta)$ can be understood as follows: let $V$ be a vector field on $M$, tangent to the boundary and satisfying $\text{curl }V=X$ (equivalently, the $1$-form $\alpha=\star i_{V} \mu$ is a primitive of $i_{X} \mu$). It is easy to see that the vector field $V$ is completely determined up to the addition of a harmonic vector field (that is, a vector field whose divergence and curl both vanish) tangent to the boundary. Fixing a helicity class is equivalent to fixing the $L^2$ projection of $V$ into the space of harmonic vector fields. The set of such harmonic vector fields is finite-dimensional and isomorphic to $H_{1}(\RR^3 \setminus M)$, so that $S(X, \mu)\cong H_{1}(\RR^3 \setminus M)$ (see \cite{bfg57}). 
\end{remark}
\subsection{Invariants of a toric flow tube} Our next goal is to compute the helicity and the Ruelle invariant of a standard family of flows. We start by introducing this family. Consider the manifold
\[
\mathbb{U} = I \times T^2, \qquad\text{with coordinates $(t,x,y)$.}
\]
Here $I$ denotes the unit interval $[0,1]$ and $T^2=S^1\times S^1$ a two-torus. We decompose the boundary of $\mathbb{U}$ into components
\[
\partial\mathbb{U} = \partial_0\mathbb{U} \sqcup \partial_1 \mathbb{U} \qquad\text{where}\qquad \partial_i \mathbb{U} = \{i\} \times T^2.
\]
We will use the tautological identifications $H^1(\mathbb{U};\mathbb{K}) = \mathbb{K}^2$ and $H^1(\partial_i\mathbb{U};\mathbb{K}) = \mathbb{K}^2$ for any coefficients $\mathbb{K}$ induced by the basis of classes $[dx]$ and $[dy]$. In particular, we adopt the notation
\[
(a,b) \in H^1(\mathbb{U};\mathbb{K}),\]
and 
\[(a_0,b_0) \oplus (a_1,b_1) \in H^1(\partial_0 \mathbb{U};\mathbb{K}) \oplus H^1(\partial_1 \mathbb{U};\mathbb{K}) = H^1(\partial\mathbb{U}).
\]
\begin{definition} \label{def:toric_flow_tube} A \emph{toric flow tube} $(X,\mu)$ is a vector field $X$ on $\mathbb{U}$ and a volume form $\mu$ which are of the form
\[
X = F(t) \cdot \partial_x + G(t) \cdot \partial_y \qquad\text{and}\qquad \mu = dt \wedge dx \wedge dy,
\]
where $F,G$ are smooth functions in $\mathbb{U}$ and $G > 0$. The pair $(F,G)$ is the \emph{profile} of $(X,\mu)$. \end{definition}

Next, we compute the helicity classes and helicity of toric flow tubes. 

\begin{lemma}[Helicity class calculation] \label{lem:helicity_classes} Let $(X,\mu)$ be a toric flow tube of profile $(F,G)$. Then the set of helicity classes in $H^1(\partial\mathbb{U};\R) = \R^2 \oplus \R^2$ is
\[
S(X,\mu) = \big\{z \oplus (z + c) \; : \; z \in \R^2\big\},
\]
where $c = \Big(\int_0^1 G dx, -\int_0^1 F dy\Big)$.
\end{lemma}

\begin{proof} Let $A,B:I \longrightarrow \R$ be the unique functions satisfying the conditions
\[
A(0) = B(0) = 0, \qquad dA = G dt,\qquad \text{and} \qquad dB = -F dt.
\]
The interior product of the vector field $X$ with the volume form $\mu$ is
\[
\iota_X\mu = G  dt \wedge dx - F  dt \wedge dy = d\alpha, \qquad\text{where}\qquad \alpha = A dx + B  dy.
\]
It is trivial to check that $[\alpha|_{\partial \mathbb{U}}] = 0 \oplus c \in \R^2 \oplus \R^2$ where $c$ is the vector above. \end{proof}

A straightforward generalization of the above argument yields the following calculation of the helicity in a fixed helicity class.

\begin{lemma}[Helicity calculation] \label{lem:helicity_calculation} Let $(X,\mu)$ be a toric flow tube of profile $(F,G)$. Fix a helicity class $\theta = (a,b) \oplus (c,d)$ and let $A,B: I \longrightarrow \R$ be the unique functions satisfying
\[
A(0) = a, \qquad B(0) = b, \qquad dA = G dt, \qquad \text{and} \qquad dB = -F  dt.
\]
Then $\alpha = Adx + Bdy$ is a primitive of $\iota_X\mu$ with $[\alpha|_{\partial M}] = \theta$ and the helicity is given by
\[
\mathcal{H}(X,\theta) = \int_0^1 (BG - AF) dt = (cd - ab) - 2\int_0^1 AF dt.
\]
\end{lemma}

Finally, we compute the Ruelle invariant of toric flow tubes. Consider the standard plane bundle spanned by $\partial_t$ and $\partial_x$, along with the tautological trivialization
\begin{equation} \label{eq:standard_normal_bundle}
E_{\on{std}} = \on{span}(\partial_t,\partial_x) \subset T\mathbb{U}, \qquad\text{and}\qquad \tau_{\on{std}}:E_{\on{std}} \longrightarrow \R^2,
\end{equation}
associated with this basis. Note that the set of homotopy classes of trivializations of $E_{\on{std}}$ forms a torsor over $H^1(\mathbb{U})$. In particular, there is a well-defined element
\[
[\sigma] - [\tau_{\on{std}}] \in H^1(\mathbb{U}), \qquad\text{for any trivialization $\sigma$ of $E_{\on{std}}$}.
\]
Via the isomorphism $[\mathbb{U},S^1] \simeq H^1(\mathbb{U})$, the class $[\sigma] - [\tau_{\on{std}}]$ can be expressed explicitly as the homotopy class of the composition bundle map $\sigma \tau^{-1}_{\on{std}}$ as an endomorphism of the trivial $\R^2$ bundle, which may be viewed as a map
\[
\sigma \tau_{\on{std}}^{-1}:\mathbb{U} \longrightarrow GL_+(2,\R) \simeq SO(2) \simeq S^1.
\]
Note that any vector field $X$ that is tranverse to $E_{\on{std}}$ has normal bundle $T\mathbb{U}/\on{span}(X)$ that is naturally isomorphic to $E_{\on{std}}$. By abuse of notation, we also let $\tau_{\on{std}}$ denote the composition
\[
T\mathbb{U}/\on{span}(X) \xrightarrow{\sim} E_{\on{std}} \xrightarrow{\tau_{\on{std}}} \mathbb{U} \times \mathbb{R}^2.
\]

\begin{lemma}[Ruelle calculation] \label{lem:Ruelle_calculation} Let $(X,\mu)$ be a toric flow tube of profile $(F,G)$ and let $\tau$ be a normal trivialization. Then the Ruelle invariant $\Ru(X,\tau)$ is given by
\[
\Ru(X,\tau) = \int_0^1 (aF + bG) dt,\]
where $[\tau] - [\tau_{\on{std}}] = (a,b)$.
\end{lemma}

\begin{proof} It suffices to show that the Ruelle density $\ru(X,\tau):\mathbb{U} \longrightarrow \R$ is given by
\[
\ru(X,\tau) = aF + bG.
\]
{\bf Step 1: Preliminaries.} We start by decomposing the Ruelle density into two pieces. Consider the linearized flows (\ref{eq:linearized_flow_in_tau}) in the trivializations $\tau$ and $\tau_{\on{std}}$ respectively.
\[\Phi:\R \times \mathbb{U} \longrightarrow \SL(2), \qquad\text{and}\qquad \Phi_{\on{std}}:\R \times \mathbb{U} \longrightarrow \SL(2).\]
We can relate these two linearized flows as follows. Let $\Psi$ denote the transition map $\tau \circ \tau^{-1}_{\on{std}}$ between the two trivializations. Then we have
\[
\Phi(T,z) = \Psi(\phi_T(z)) \Phi_{\on{std}}(T,z)\Psi(z)^{-1} \in \SL(2).
\]
We will require a slightly different formula. Let $U:\R \times \mathbb{U} \longrightarrow \SL(2)$ and $Q:\R \times \mathbb{U} \longrightarrow \SL(2)$ be the maps
$$U(T,z) = \Psi(\phi_T(z))\Psi(z)^{-1},$$ 
and 
$$Q(T,z) = \Psi(z) \Phi_{\on{std}}(T,z)\Psi(z)^{-1}.$$
Note that $U$ and $Q$ are the identity along $0 \times \mathbb{U}$, and therefore both maps admit lifts
\[
\widetilde{U}:\R \times \mathbb{U} \longrightarrow \twSL(2) \qquad\text{and}\qquad \widetilde{Q}:\R \times \mathbb{U} \longrightarrow \twSL(2),
\]
that are the identity along $0 \times \mathbb{U}$. By the uniqueness of such lifts, we know that the corresponding lift of $\Phi$ is given by
\[
\widetilde{\Phi}(T,z) = \widetilde{U}(T,z)\widetilde{Q}(T,z).
\]
By the construction of the Ruelle density and the quasimorphism property of the rotation quasimorphism $\rho$ \cite[Def 2.3]{ce2022}, we have
\begin{equation} \label{eq:ru_density_UQ}
\ru(X,\tau) = \lim_{T \longrightarrow \infty} \frac{\rho(\widetilde{U}_T\widetilde{Q}_T)}{T} = \lim_{T \longrightarrow \infty} \frac{\rho(\widetilde{U}_T)}{T} + \lim_{T \longrightarrow \infty} \frac{\rho(\widetilde{Q}_T)}{T}. 
\end{equation}
We compute the contribution of each of these parts in two steps below.

\vspace{3pt}

\noindent {\bf Step 2: Shear part.} We compute the contribution of $Q$ to the formula (\ref{eq:ru_density_UQ}). Note that the flow $\phi:\R \times \mathbb{U} \longrightarrow \mathbb{U}$ of the vector field $V$ is given by
\[
\phi_T(t,x,y) = (t,x + T \cdot F(t), y + T \cdot G(t)).
\]
In particular, the differential of the flow preserves the vector field $\partial_t$ for all $T$. This implies that $\Phi_{\on{std}}(T,z)$ has a $1$ eigenvalue for all $(T,z)$, and similarly for $Q$ since it is conjugate to $\Phi_{\on{std}}$. Now we simply note that $Q(T,z)$ must be conjugate to a positive shear matrix for all $(T,z)$ and
\[
\rho(\widetilde{A}) \in \Z \qquad\text{for any lift $\widetilde{A}$ of a positive shear matrix (cf. \cite[\S 2.1]{ce2021}).}
\]
Thus $\rho \circ \widetilde{Q}$ vanishes since is continuous, integer-valued, and equal to $0$ when $T = 0$. Therefore, we conclude that
\[
 \lim_{T \longrightarrow \infty} T^{-1} \cdot \rho(\widetilde{Q}_T) = 0.
\]
{\bf Step 3: Unitary Part.} We compute the contribution of $U$ to the formula (\ref{eq:ru_density_UQ}). Since the Ruelle density is invariant up to homotopy of the trivialization, we may choose $\tau$ so that
\[
\Psi = \tau\tau^{-1}_{\on{std}} \qquad\text{is given by}\qquad \Psi(t,x,y) = 
\exp\big(2\pi(ax + by) \cdot J\big),
\]
where
$$\quad J = \left[\begin{array}{cc}
0 & -1\\
1 & 0
\end{array}\right].$$
In other words, $\Psi(t,x,y)$ is simply the rotation matrix by $2\pi(ax + by)$ radians. It is simple to see that $\Psi$ represents the cohomology class $(a,b)$. Indeed, for maps $u:\mathbb{U} \longrightarrow S^1$, the corresponding pair of integers specifying the cohomology class can be expressed as follows.
\[
(a,b) \qquad\text{where}\qquad a = \on{deg}(u|_{0_t \times S^1_x \times 0_y}) \qquad\text{and}\qquad b = \on{deg}(u|_{0_t \times 0_x \times S^1_y})
\]
With this observation, and by computing the corresponding degrees of $\Psi$, we see that $\Psi \circ \tau_{\on{std}}$ satisfies $[\Psi \circ \tau_{\on{std}}] - [\tau_{\on{std}}] = (a,b)$ and thus is in the isotopy class of $\tau$. Next, by direct computation
\[
U(T,z) = \Psi(\phi_T(z))\Psi(z)^{-1} = \exp\big(2\pi(aF(t) + bG(t))T \cdot J\big),
\]
where $z = (t,x,y)$.
Finally, the rotation quasimorphism agrees with the standard identification $\widetilde{U}(1) \longrightarrow \R$ on the universal cover of the unitary group $\widetilde{U}(1) \subset \twSp(2)$ (cf. \cite[Def 2.3]{ce2022}). It follows that
\[
 \lim_{T \longrightarrow \infty} T^{-1} \cdot \rho(\widetilde{U}_T) = \lim_{T \longrightarrow \infty} T^{-1} \cdot (aF(t) + bG(t))T = aF(t) + bG(t).
\]
This concludes the computation of $\ru(V,\tau)$ via (\ref{eq:ru_density_UQ}) and thus the proof. \end{proof}

We conclude this part with a useful lemma about trivializations of $E^{\on{std}}$ (or equivalently, of the normal bundle to the vector field $X$ for any toric flow tube $(X,\mu)$). Consider the embedding
\begin{align} \label{eq:toric_to_tube_map} 
\begin{split}
\jmath:\mathbb{U} &\longrightarrow  S^1 \times D\\
(t,x,y) &\longmapsto (y,\frac{1 + t}{2} \cdot e^{2\pi i x}).
\end{split}
\end{align}
Under this map, $E_{\on{std}} \subset T\mathbb{U}$ extends to a sub-bundle $E_{\on{std}} \subset T(S^1 \times D)$ as the tangent bundle $TD$ in each fiber of $T(S^1 \times D)$. We claim the following lemma.

\begin{lemma} \label{lem:trivialization_lemma} Let $\tau$ be a trivialization of $E_{\on{std}}$ that is the pullback of a trivialization over $S^1 \times D$. Then we have
\[
[\tau] - [\tau_{\on{std}}] = (1,a) \in \Z^2 \simeq H^1(\mathbb{U};\Z).
\]
\end{lemma}

\begin{proof} The induced map on 1st cohomology of $\jmath:\mathbb{U} \longrightarrow S^1 \times D$ is identified with the map $\Z \longrightarrow \Z^2$ given by $a \longmapsto (0,a)$ under the identification induced by the classes $[dx]$ and $[dy]$. Moreover, pullback intertwines the $H^1$-action on trivializations in the sense that
\[
\jmath^*(\tau + \theta) = \jmath^*\tau + \jmath^*\theta.
\]
Thus it suffices to show that $[\tau] - [\tau_{\on{std}}] = (1,0)$ for some pullback trivialization $\tau = \jmath^*\sigma$. We choose the trivialization
\[
\sigma:E_{\on{std}} = TD^2 \simeq \R^2 \qquad\text{of the bundle} \qquad E_{\on{std}} \longrightarrow S^1 \times D.
\]
The transition map $\Psi = \jmath^*\sigma \tau^{-1}_{\on{std}}$ as a map $\mathbb{U} \longrightarrow \on{End}(\R^2)$ is simply the differential of the transition map between $(t,x)$-coordinates on $I \times S^1$ and the coordinates on $D$. That is, we have
\[
\Psi(t,x,y) = \frac{1}{2}
\left[\begin{array}{cc}
\cos(2\pi x) & -\pi(t+1)\sin(2\pi x)\\
\sin(2\pi t) & \pi(t+1)\cos(2\pi x)
\end{array}\right].
\]
Now note that $\Psi$ is isotopic to the family $\mathbb{U} \longrightarrow U(1) = SO(2)$ given by $\Psi'(t,x,y) = e^{2\pi i x}$. It follows that $\Psi$ represents the cohomology class $(1,0)$ and thus that $[\jmath^*\sigma] - [\tau_{\on{std}}] = (1,0)$. \end{proof}

\subsection{Construction of toric flow tubes} We next show that any $C^1$-open set not contained in the set of Anosov vector fields contains a vector field with an embedded toric flow tube. We start by recalling the following terminology.

% \begin{definition}\label{def:farAno} A vector field $X$ on a manifold $M$ is \emph{far-from-Anosov} if it admits a $C^1$-neighborhood containing no Anosov vector field. 
% \end{definition}

\begin{thm}[Embedded toric tube] \label{thm:toric_flow_tube} For any connected $C^1$-open set $\mathcal{U} \subset \mathfrak{X}_\mu(h)$ containing a non-Anosov vector field, there is a vector field $X \in \mathcal{U}$ and a smooth embedding
\[
\iota:\mathbb{U} = I \times T^2 \longrightarrow M
\]
such that $(\iota^*\mu, \iota^*X)$ is a toric flow tube in the sense of Definition \ref{def:toric_flow_tube}. Moreover, there is an embedding
\[
\kappa:S^1 \times D \longrightarrow M 
\]
such that $\iota = \kappa \circ \jmath$ where $\jmath:\mathbb{U} \longrightarrow S^1 \times D$ is the map in (\ref{eq:toric_to_tube_map}) and such that $\kappa^*X$ is transverse to the sub-bundle $TD \subset T(S^1 \times D)$.\end{thm}
\begin{remark}
    It follows from the proof that $\iota^*X$ can be chosen to have a Diophantine rotation vector on each torus fiber.
\end{remark}
\begin{proof}

\textbf{Step 1. A vector field in $\mathcal{U}$ with an elliptic periodic orbit.}
Let $X$ be a non-Anosov vector field in $\mathcal{U}$. To add an elliptic orbit to $X$ by a $C^1$-perturbation, we can directly apply Theorem \ref{thm:fixedhelliptic}. This is strictly the only option when $h=0$. However, when $h\neq 0$, one can directly use \cite[Theorem 1]{BeDu}, corrected by our Lemma~\ref{lem:franks}, in the following way. An application of the theorem tells us that there is some $\tilde X$ in that is arbitrarily close to $X$ in $\mathfrak{X}_\mu(M)$ with an elliptic periodic orbit, but the helicity of $\tilde X$ is in general different from $h$. However, suppose we choose $\tilde X$ sufficiently close to $X$. In that case, the helicity $\tilde h$ of $\tilde X$ will be nonvanishing and the vector field $\hat X={\frac{\sqrt{h}}{\sqrt{\tilde h}}}\tilde X$ will be in $\mathcal{U}$ (and will still have an elliptic periodic orbit). We rename $\hat X$ back to $X$.
\begin{remark}
It is worth noting that, also when $h\neq 0$, one can use instead of \cite[Theorem 1]{BeDu} a "folklore" theorem of Newhouse \cite[Theorem 6.2]{New1977}, which states that given a Hamiltonian vector field with a non-Anosov energy level set, it is possible to perturb the Hamiltonian function so that the energy level set admits an elliptic period orbit. This is done by taking a framing one-form $\lambda$ for the Hamiltonian structure $\iota_X\mu$, considering a symplectization of $(\lambda,\omega)$, and applying Newhouse's theorem. However, Newhouse's result has not been formally written in the literature except for the case of a 4D symplectic manifold \cite{BD}, which again relies on the wrong lemma in \cite{AM}, and that can be fixed using our Lemma \ref{lem:franks}. Formalizing Newhouse's theorem in general might require adapting Lemma \ref{lem:franks} to the Hamiltonian setting.
    % Consider the symplectic manifold 
% $$M\times (-\varepsilon,\varepsilon) \qquad \text{endowed with the form} \qquad  \omega=d\alpha + d(t\lambda),$$
% where $d\alpha=\iota_X\mu$, the one-form $\lambda$ is a framing one-form of $d\alpha$ (i.e. $\lambda\wedge d\alpha>0$ and $\lambda(X)=1$), and $t$ is a coodinate in $(-\varepsilon,\varepsilon)$. The Hamiltonian function $H=t$ satisfies that its Hamiltonian vector field along $t=0$ coincides with $X$. By \cite[Theorem 6.2]{New1977}, there is a Hamiltonian $\tilde H$ (possibly of $C^2$-regularity) that is $C^2$-close to $H$ so that the Hamiltonian vector field along $\{\tilde H=0\}$ admits an elliptic periodic orbit. We can perturb $\tilde H$ so that it is smooth and the vector field it defines still has an elliptic periodic orbit, we keep denoting by $\tilde H$ this smooth function. The two-form $d\tilde\alpha$ obtained by pulling back $\omega$ to $\{\tilde H=0\}$ is $C^\infty$-close to $d\alpha$, and the vector field $\tilde X$ defined by $\iota_{\tilde X}\mu=d\tilde \alpha$ is an exact vector field that is $C^\infty$-close to $X$ and that admits an elliptic periodic orbit. We rename $(\tilde X, \tilde \alpha)$ back to $(X,\alpha)$.
\end{remark}

\textbf{Step 2. Linearization of the first-return map.}
A next step is the following lemma, which is based on \cite[Proposition 47]{C}, combined with \cite[Lemma 5.5]{Edt}. 

\begin{lemma}\label{lem:ellipticlinear}
    Let $X$ be an (exact) volume-preserving vector field in $M$, and $\gamma$ an elliptic periodic orbit of $X$. Then there exists a $C^2$-close exact volume-preserving vector field $\tilde X$ satisfying:
    \begin{enumerate}
        \item $\tilde X$ coincides with $X$ away from the neighborhood of a point $p\in \gamma$,
        \item $\gamma$ is a periodic orbit of $\tilde X$,
        \item the first return map of $\tilde X$ in a small enough transverse section near $p$ is conjugate to an irrational rotation with Diophantine rotation number.
    \end{enumerate}
\end{lemma}
\begin{proof}
Consider the symplectic manifold
$$\left(M\times (-\varepsilon, \varepsilon), \omega=d(t\lambda) + d\alpha\right),$$ 
where $\lambda$ is a one-form such that $\lambda(X)=1$ and $\iota_X\mu=d\alpha$. The vector field $X$ corresponds to the Hamiltonian vector field of $H=t$ along $\{H=0\}$. Abusing notation, we identify $M$ with $M\times\{0\} \subset M\times (-\varepsilon,\varepsilon)$. Pick a point $p\in \gamma$ and choose Hamiltonian flow-coordinates $(x,y,u,v)$ on a suitable open neighborhood $$U\cong D\times [-1,1]_u \times (-\delta,\delta)_v\subset M\times (-\varepsilon,\varepsilon),$$
where $D$ is a disk of radius one. Namely, one has
$$H=v, \qquad X_H= \pp{}{u}, \qquad \text{and} \qquad \omega= dx\wedge dy + du\wedge dv.$$
In these coordinates, the periodic orbit along $U$ is given by  
$$\gamma \cap U=\{0\}\times (-1,1)\times \{0\}.$$ Following the notation of \cite[Proposition 47]{C}, there is a small disk $D_0\subset D$ such that the first return map of $X$ along $D_0\times \{1\} \times \{0\}$ is well defined, with image on $D\times \{-1\} \times \{0\}$. The first return map $$F: D_0 \longrightarrow D$$ 
defines a symplectomorphism from $D_0$ to its image, fixing the origin, which is an elliptic fixed point of $F$. Now \cite[Lemma 5.5]{Edt} implies that there is a compactly supported time-dependent Hamiltonian function
$$K=K(x,y,s) \in C^\infty(D\times (0,1)),$$ 
satisfying the following properties:
\begin{itemize}
    \item[-] its support can be chosen in an arbitrarily small neighborhood of the origin of the disk $D$,
    \item[-] the $C^2$-norm of the Hamiltonian vector field $X_K$ can be chosen to be arbitrarily small,
    \item[-] the $C^3$-norm of $K$ is arbitrarily small,
    \item[-] if $\phi_K^1$ is the time one-map of $X_k$, the symplectomorphism $\phi_K^1\circ F$ fixes the origin and is conjugated to an irrational rotation in a small enough neighborhood of it.
\end{itemize}
 Consider in $M\times (-\varepsilon,\varepsilon)$ the Hamiltonian function $\tilde H= v+ {K(x,y,u)}$. As argued in \cite[Proposition 47]{C}, the Hamiltonian vector field of $\tilde H$ along $\{\tilde H=0\}$ is an exact volume-preserving vector field $\tilde X$ that is $C^2$-close to $X$ and has as first return map near $p$ the symplectomorphism $\phi_K^q\circ F$, which is conjugate to a rigid rotation.\\
 
 If we want the rotation to have a Diophantine rotation number, we make an additional perturbation to the recently constructed vector field, which we will denote by $X$ again. Take another small enough Hamiltonian flow-box neighborhood $D\times [-1,1]_u \times (-\varepsilon,\varepsilon)_v$ of $p$ where the first return map along $D\times \{0\}$ is a rigid irrational rotation. Consider a $C^\infty$-small compactly supported Hamiltonian $h\in C^\infty(D\times [-1,1] {\times (-\varepsilon, \varepsilon)})$ of the form 
 $$h(r,{u})=\phi(r)\varphi({u}), \quad  \text{where } (r,\theta) \text{ are polar coordinates in } D.$$
 We choose the function $\varphi({u})$ to be a bump function equal to zero near ${u}=\pm 1$, and we choose $\phi(r)$ with compact support such that $\pp{\phi}{r}\equiv ar$ for a small enough radius and a small enough constant $a>0$. If we do as before and consider the Hamiltonian vector field of $v+h$ along the zero level set, we obtain an exact vector field that is $C^\infty$-close to $X$, has $\gamma$ as a periodic orbit, but the first-return map near $p$ is the composition of the first-return map of $X$ with the symplectomorphism $\psi^1_h$ generated by $h$. Varying $a$, this symplectomorphism is a rotation of different angles, and thus for a suitable $a$, the first return map of $X$ composed with $\psi^1_h$ is a rotation with Diophantine rotation number. This concludes the proof of the lemma.
\end{proof}

\textbf{Step 3. Finding a toric flow tube.} Choose a small disk $D\subset M$ transverse to the flow centered at $p\in \gamma$, where the first return map is given by an irrational rotation $R_\alpha:D\longrightarrow D$. Denote by 
$$T: D \longrightarrow D,$$
the first-return time function of the flow along $D$. A neighborhood of $\gamma$ is given by $V/\sim$, where 
$$V=\{ (p,z)\in D\times \mathbb{R} \mid z\in [0,T(p)]\},$$ 
and we identify $(p,0)$ with $(R_\alpha(p), T(p))$.  The vector field is then written as $X=\pp{}{z}$. 

This shows that up to reparametrizing $X$ so that the first return time is equal to $1$, the flow of $X$ is smoothly conjugate to the suspension of an irrational rotation of the disk. The latter is smoothly conjugate to a linear flow in the torus. Thus, there exists coordinates $(r,\theta,s)$ of $D^2\times S^1$ and a diffeomorphism 
$$\phi : D^2\times S^1 \longmapsto V,$$
such that 
$$ \phi^*X= f(r,\theta,s)\left( \pp{}{s} + \alpha\pp{}{\theta} \right). $$
The pullback of the volume form $\mu$ will in general be of the form $H(r,\theta,s)ds\wedge rdr\wedge d\theta$.
\medskip

To obtain an improved normal form, we will use the fact that the rotation number of the disk is Diophantine, which will help us linearize the flow on each invariant torus of fixed radius (after possibly removing a smaller neighborhood of the core orbit $\{r=0\}$). Consider the subset $V'\subset V$ given by those points in $V$ whose radius $r$ in the $D^2$ factor belongs to some closed interval $[\delta_0, \delta_1]$, with $0<\delta_0<\delta_1$. For each possible value of the radius $r\in [\delta_0,\delta_1]$, the flow is identified with a suspension flow with non-constant return time $F_r(\theta)=T(r,\theta)$. We aim to find a new cross-section of the flow in $V'$ for which the first-return time is constant. Along a given torus $\{r=C\}$, this can be done as in \cite[Proposition 2.9.5]{Katok_Hasselblatt_1995}: we try to find a function $G_r(\theta)$ whose graph in $\{(\theta,z)\in S^1\times \mathbb{R} \mid z\in [0, F_r(\theta)]\}/\sim$ defines a suitable new section. The function $G_r(\theta)$ is obtained by solving the equation
$$G_r(\theta+\alpha) - G_r(\theta) = \int_{S^1} F_r(\theta) d \theta - F_r(\theta), $$
which can be solved by considering the Fourier expansion of both functions and using that $\alpha$ (the rotation number) is Diophantine. Notice at this point that this process can be done parametrically in $r$, since Fourier coefficients of $F_r(\theta)$ vary continuously with $r$. The first-return time of the new section, given along each torus by $\int_{S^1}F_r(\theta)$, varies with the radius $r$. Using this new section, we can find a diffeomorphism

$$\psi:[\delta_0,\delta_1]\times T^2 \longrightarrow V',$$
satisfying
$$\psi^*X= f(r) \left( \pp{}{s} + \alpha \pp{}{\theta}\right).$$
A priori the volume form in these new coordinates is
$$\widetilde \mu=K(r,\theta,s)ds\wedge dr \wedge d\theta,$$ 
where $K$ is some positive function. However, notice that $X$ preserves not only $\widetilde \mu$, but also $ds\wedge dr\wedge d\theta$. In particular, taking the exterior derivative of $\iota_X\widetilde \mu$ we deduce that $dK\wedge \iota_X\mu=0$ and thus that $\iota_XdK=0$. Namely, the function $K$ is a first integral of $X$. However, the vector field $X$ is minimal in each invariant torus $\{r=\operatorname{ct}\}$, and thus $K=K(r)$. Up to doing a change of coordinates of the form $\tilde r(r)$, we can then assume that $\hat \mu= ds\wedge rdr\wedge d\theta$, and thus $V'$ with $\hat \mu$ and the vector field $X$ after the change of coordinates $\tilde r$ define a toric flow tube. By construction, the toric flow tube is the image of a map of the form \eqref{eq:toric_to_tube_map} into a domain $U$ of $M$ (a neighborhood of the elliptic orbit constructed in Lemma \ref{lem:ellipticlinear}) which is the image of an embedding of the form
$$\kappa: U \longmapsto S^1\times D.$$ 
Since $X$ is transverse to the disk fibers of $U$ transverse to the elliptic orbit, this implies that $\kappa^*X$ is transverse to $TD$.\\

We have thus shown that given a non-Anosov vector field $X\in \mathcal{U}$, we can find a vector field $Y \in \mathfrak{X}_\mu(M)$ that is arbitrarily close to $X$ and admits a toric flow tube (with a map $\kappa$ as described in the statement of the proposition, notice that since {the rotation vector is Diophantine, we can always assume that $G>0$} as required in Definition \ref{def:toric_flow_tube}). The helicity of $Y$, however, might not be $h$ anymore. {If $h\neq 0$}, we can rescale the vector field to $X'=\frac{\sqrt{h}}{\sqrt{\tilde h}}Y$, where $\mathcal{H}(Y)=\tilde h$ can be assumed to be non-zero as well.\\

In the special case $h=0$ (for which, we recall, in Step 1 -- the perturbation creating an elliptic orbit -- we had to use Theorem \ref{thm:fixedhelliptic} necessarily), we use Corollary \ref{cor:localmod} instead to correct the helicity in a neighborhood $V$ that we fix at the beginning of Step 2 and that is disjoint from a neighborhood of the elliptic periodic orbit. We can then produce the toric flow-tube disjoint from $V$ by a small enough perturbation, and then correct helicity by a perturbation with compact support in $V$.
\end{proof}

\subsection{The Ruelle map near non-Anosov flows} We conclude this section with a proof of the main result of the section, Theorem \ref{thm:main_ruelle_prop} below, which corresponds to the first item in our main Theorem \ref{thm:main}. We start by fixing
\[
\text{a homology $3$-sphere $M$} \qquad\text{and}\qquad \text{a volume form $\mu$}.
\]
Since $M$ is a homology $3$-sphere, the normal bundle $TM/\on{span}(X)$ of any non-vanishing vector field $X$ admits a unique trivialization up to isotopy. Thus the Ruelle invariant constructed in Section \ref{subsec:Ruelle_construction} yields a continuous function
\[
\Ru: \mathfrak{X}_{\mu,1}(h) \longrightarrow \R
\]
on the space of differentiable, nowhere vanishing, $\mu$-preserving vector fields on $M$ with helicity $h$, equipped with the $C^1$-topology. Our final goal in this section is to prove the following result.

\begin{thm} \label{thm:main_ruelle_prop} Let $\mathcal{U} \subset \mathcal{V}^1_h(M,\mu)$ be an open set containing some non-Anosov vector field $X$. Then
\[\Ru|_{\mathcal{U}} \quad\text{is non-constant.}\]\end{thm}

\begin{proof} We construct a smooth family of vector fields $X_\varepsilon$ of constant helicity whose Ruelle invariants cover an open set in $\R$. We break the construction into a few steps.

\vspace{3pt}

{\bf Step 1: Toric tube.} By Theorem \ref{thm:toric_flow_tube}, there is a vector field $X \in \mathcal{U}$ and an embedding
\[
\iota:\mathbb{U} = I \times T^2 \longrightarrow M \qquad\text{with image}\qquad N \subset M,
\]
such that $(\iota^*X,\iota^*\mu)$ is a toric flow tube. In particular, in coordinates $(t,x,y)$ on $I \times T^2$, the vector field $X$ and the volume form $\mu$ pull back to 
\begin{equation} \label{eq:prop:main_ruelle_prop1}\iota^*X = F(t) \cdot \partial_x + G(t) \cdot \partial_y \quad\text{and}\quad \iota^*\mu = dt \wedge dx \wedge dy,\end{equation}
where $G > 0$. We track some helicity class and trivialization data to calculate helicity and the Ruelle invariant. 

\vspace{3pt}

First, let $\theta = (a,b) \oplus (c,d) \in H^1(\mathbb{U};\R)$ denote the unique helicity class of $(\iota^*X,\iota^*\mu)$ given by
\[
\theta = [\iota^*\beta|_{\partial \mathbb{U}}]\qquad\text{for any primitive $\beta$ of $\iota_X\mu$ on $M$.}
\]
This is independent of the primitive since $M$ is a homology $3$-sphere. Note that any primitive $\beta$ of $\iota^* \iota_X \mu$ in the helicity class $\theta$ extends to a primitive over $M$. Thus by Lemma \ref{lem:helicity_calculation}, we may choose a primitive $\alpha$ of $\iota_X\mu$ such that
$$\iota^*\alpha = A dx + B dy,$$
where
\begin{equation} \label{eq:prop:main_ruelle_prop2}
A(0) = a, \qquad B(0) = b, \qquad dA = Gdt, \quad \text{and}\quad dB = -Fdt.
\end{equation}
Next, fix a trivialization 
\[\tau:TM/\on{span}(X) \simeq \R^2\]
of the normal bundle of the vector field $X$ over all of $M$.
 This pulls back to a trivialization $\iota^*\tau$ of $T\mathbb{U}/\on{span}(\iota^*X) \simeq E_{\on{std}}$ (see (\ref{eq:standard_normal_bundle})) over $\mathbb{U}$. By Theorem \ref{thm:toric_flow_tube} that $\iota = \kappa \circ \jmath$, where $\jmath$ is the map (\ref{eq:toric_to_tube_map}) and $\kappa:S^1 \times D \longrightarrow M$ is a map whose disk fibers are everywhere transverse to the vector field $X$. This implies that $\iota^*\tau$ factors through a trivialization of $E_{\on{std}}$ over $\mathbb{R}/\mathbb{Z} \times D$. By Lemma \ref{lem:trivialization_lemma}, this implies that
\begin{equation} \label{eq:prop:main_ruelle_prop3}
[\iota^*\tau] - [\tau_{\on{std}}] = (1,b) \in \Z^2 = H^1(\mathbb{U};\Z),
\end{equation}
where $\tau_{\on{std}}$ is the tautological trivialization of $E_{\on{std}}$ (see (\ref{eq:standard_normal_bundle})). 

\vspace{3pt}

{\bf Step 2: Family Of Vector Fields.} We next construct a family $X_\varepsilon$ of $C^1$ vector fields depending smoothly on a parameter $\varepsilon$. We will require the following elementary lemma.

\begin{lemma}\label{lem:L2_lemma} Let $g$ and $h$ be linearly independent functions in $L^2(0,1)$. Then there existe a compactly supported smooth function $f$ such that $\langle f,g\rangle_{L^2} = 0$ and $\langle f,h\rangle_{L^2} = 1$.
\end{lemma}

%Thus the constant function $1$ and $A$ are linearly independent in $L^2$ with respect to the measure $\nu = Hdr$ and there must be functions\[g, h \in L^2(I;\nu) \qquad\text{such that}\qquad \langle g,A\rangle_{L^2} = \langle h,1\rangle_{L^2} = 1 \text{ and }\langle g,1\rangle_{L^2} = \langle h,A\rangle_{L^2} = 0\]By approximating $g$ and $h$ by compactly supported smooth functions, we get compactly supported smooth functions $g'$ and $h'$ such that\[\langle g',A\rangle_{L^2} \text{ and } \langle h',1\rangle_{L^2} \text{ are near $1$}\qquad \text{and}\qquad \langle g',1\rangle_{L^2} \text{ and } \langle h',A\rangle_{L^2}\text{ are near $0$}\]The desired function $f$ can be constructed as a linear combination of $g'$ and $h'$. 

\noindent Since $A$ is non-constant, it is linearly independent from the constant function equal to $1$. By applying Lemma \ref{lem:L2_lemma} to the function $A$ and $1$ and the $L^2$-inner product induced by the measure $ dr$, we can thus choose a function $f:I \longrightarrow \R$ such that: 
\begin{itemize}
    \item[-] $f(r)= 0$ near $r=0$ and $r=1$,
    \item[-] $\int_0^1 fA dr = 0$,
    \item[-] $\int_0^1 f dr = 1$.
\end{itemize}
We take $F_\varepsilon = F + \varepsilon f$ and define the vector field $V_\varepsilon$ as follows.
\[
\iota^*X_\varepsilon = F_\varepsilon(r) \cdot \partial_s + G(r) \cdot \partial_t \qquad\text{and}\qquad X_\varepsilon = V \text{ on }M \setminus \iota(\mathbb{U}).
\]

{\bf Step 3: Helicity and Ruelle of family.} We next compute the helicity and the Ruelle invariant of the family of vector fields. From the construction of helicity (see Section \ref{subsec:helicity}) we see that
\[
\mathcal{H}(X_\varepsilon) = \mathcal{H}(X_\varepsilon|_{M \setminus N},\theta) + \mathcal{H}(X_\varepsilon|_N,\theta) = \mathcal{H}(X|_{M \setminus N},\theta) + \mathcal{H}(\iota^*X_\varepsilon,\theta) 
\]
Note that the helicity class $\theta$ is independent of $\varepsilon$ since the vector-field and volume form are independent of $\varepsilon$ outside of $N$. From the computation of the helicity in Lemma \ref{lem:helicity_calculation}, we see that
\[
\mathcal{H}(\iota^*X_\varepsilon,\theta) - \mathcal{H}(\iota^*X,\theta) = -2\varepsilon \cdot \Big(\int_0^1 f \cdot A\Big) = 0
\]
It follows that the helicity of $X_\varepsilon$ is equal to that of $X$ and thus $\mathcal{H}(X_\varepsilon) = h$. To compute the Ruelle invariant, we apply Proposition \ref{prop:Ru_disjoint} to see that 
\begin{align*}
    \Ru(X_\varepsilon) &= \Ru(X_\varepsilon,\tau)\\
    &= \Ru(X_\varepsilon|_{M\setminus N},\tau|_N) + \Ru(X_\varepsilon|_N,\tau|_N)\\
    &= \Ru(X|_{M\setminus N},\tau|_N) + \Ru(\iota^*X_\varepsilon,\iota^*\tau)
\end{align*}

Here $\tau$ is a global trivialization of the normal bundles of $X_\varepsilon$ (which are all identified for $\varepsilon$ small). By (\ref{eq:prop:main_ruelle_prop3}) and the Ruelle computation in Lemma \ref{lem:Ruelle_calculation}, we see that
\begin{align*}
    \Ru(\iota^*X_\varepsilon,\iota^*\tau) &= \int_0^1 (F + \varepsilon f + bG) dr\\
    &= \Ru(\iota^*X,\iota^*\tau) + \varepsilon \int_0^1 f dr\\
    &= \Ru(\iota^*X,\iota^*\tau) + \varepsilon
\end{align*}
Therefore the Ruelle invariant of $X_\varepsilon$ is $\Ru(X) + \varepsilon$.  

\vspace{3pt}

{\bf Step 4: Conclusion.} Finally, we simply note that $X_\varepsilon$ converges to {$X$} in $C^1$ as $\varepsilon$ goes to zero. It follows that there is a $\delta > 0$ such that {$X_\varepsilon \in \mathcal{U}$} for all $\varepsilon$ with $|\varepsilon| < \delta$. Then by the calculation of the Ruelle invariant above, we know that the open interval
\[
(-\delta + \Ru(X),\Ru(X) + \delta)\] 
is contained in $\Ru(\mathcal{U})$.
\end{proof}

\section{Topological entropy of Anosov vector fields}\label{sec:Anosov}

In this section, we discuss the topological entropy of a volume-preserving Anosov vector field. The main goal is to prove Proposition \ref{prop:main_entropy_prop} which states that topological entropy is everywhere non-constant among Anosov vector fields of any fixed value of helicity. The main tool is a formula for the first and second derivatives of the topological entropy for Anosov flows computed in \cite{katok1990differentiability, pollicott1994derivatives}.

\subsection{Anosov flows} We start by reviewing some basic facts about Anosov vector fields. We refer the reader to Fisher-Hasselblatt \cite{fh2019hyperbolic} for a more thorough discussion.

\begin{definition} A vector field $X$ on a manifold $M$ generating a flow {$\phi^{t}_{X}$} is called \emph{Anosov} if there is a continuous and $\phi^{t}_{X}$-invariant splitting
\[
TM = \on{span}(X) \oplus E^u \oplus E^s
\]
where $E^u$ is uniformly expanding and $E^s$ is uniformly contracting, in the sense that for any Riemannian metric $g$ on $M$ there exists $C,\lambda > 0$ such that
\[\|D\phi^{t}_{X}|_{E^s}\|_g \le C e^{-\lambda t} \qquad\text{and}\qquad \|D\phi^{t}_{X}|_{E^u}\|_g \ge C e^{\lambda t}\]
\end{definition}

We will need the following standard facts about Anosov vector fields, whose proofs may be found in \cite{fh2019hyperbolic}. These are stated as in Pollicott \cite{pollicott1994derivatives}.

\begin{thm}[Openness] The set of Anosov vector fields on a closed manifold $M$ is $C^1$-open. 
\end{thm}

\begin{thm}[Structural stability] \label{thm:structural_stability} Let $X_s$ be a smooth family of smooth Anosov vector fields on a closed manifold $M$. Then there is a unique family of continuous maps and functions
\[
\Psi_s:M \longrightarrow M \qquad\text{and}\qquad F_s:M \longrightarrow \mathbb{R} \qquad\text{with $F_0 = 1$ and $\Psi_0 = \on{Id}$}
\]
such that $\Psi_s^*{\phi^{t}_{X_s}} = \Psi_{{s}}^{-1} \circ {\phi^{t}_{X_s}} \circ \Psi_s$ is the flow generated by the vector field $F_s \cdot X_0$. Moreover, the families $F$ and $\Psi$ vary smoothly in the parameter $s$. \end{thm}

\subsection{Entropy} We next review the notions of topological entropy and measure entropy of a dynamical system with an emphasis on the properties of entropy in the Anosov case.

\vspace{3pt}

Topological entropy tracks the rate at which a dynamical system spreads points around. To be precise, fix a vector field $X$ with flow $\Phi$ on a manifold $M$ with Riemmannian metric $g$. Consider the values
\[
N_{X,g}(T,\delta) = \on{max}\big\{N \; : \;  \text{$\exists p_1\dots p_N \in M \mid \underset{i \neq j}{\on{min}} \big(\underset{t \in [0,T]}{\on{max}} \big(d_g(\phi^{t}_{X}(p_i),\phi^{t}_{X}(p_j)\big) \big) > \delta$}\big\},
\]
which can be used to give the following definition of topological entropy due to Bowen \cite{bowen1970topological}.

\begin{definition} The \emph{topological entropy} $\TopEn(X)$ is defined as
\[
\TopEn(X) := \lim_{\delta \longrightarrow 0} \Big(\limsup_{T \longrightarrow \infty} \frac{\log(N_{X,g}(T,\delta))}{T}\Big).
\]
\end{definition}

The entropy has many special properties in the Anosov case that will play a key role in the proof of Proposition \ref{prop:main_entropy_prop}. For example, we note that the entropy is positive in the Anosov case.

\begin{thm} \label{thm:Anosov_positive_entropy} The topological entropy of an Anosov vector field $X$ on a manifold $M$ is positive and finite.
\end{thm}

Another key property of the entropy in the Anosov setting is the following smoothness result due to Katok-Knieper-Pollicott-Weiss \cite[Thm 2]{katok1990differentiability}.

\begin{thm} \cite{katok1990differentiability} \label{thm:smoothness_of_entropy} 
For any integer $k\geq 1$, the topological entropy is a $C^{k-1}$-function on the space of Anosov vector fields of $C^r$ regularity (with $k\leq r\leq \infty$) endowed with the $C^k$-norm.
\end{thm}

\noindent Moreover, a precise formula for the first and second derivatives of the topological entropy was provided by Katok-Knieper-Policott-Weiss \cite{katok1990differentiability} and Pollicott \cite[Thm 1]{pollicott1994derivatives} respectively. In order to state this result, we require a few more definitions. 

\vspace{3pt}

Recall that a measure $\mu$ \emph{maximizes entropy} for a vector field $X$ if the measure theoretic entropy $\En(X,\mu)$ (cf. \cite[A.2.27]{fh2019hyperbolic}) is equal to the topological entropy $\En(X)$. These measures are unique up to scaling for transitive Anosov flows (cf. \cite[\S 7]{fh2019hyperbolic}). We use the following definition.

\begin{definition} \label{def:BM_measure}The \emph{Bowen-Margulis measure} $\mu_{\on{BM}}$ on a closed manifold $M$ with a transitive Anosov vector field $X$ is the unique invariant probability measure that maximizes entropy.
\end{definition}

\noindent We also recall the notion of the variance of a function with respect to a measure preserving flow. We adopt the definition of Pollicott \cite[\S 4]{pollicott1994derivatives}.

\begin{definition} Let $X$ be a flow on a closed manifold $M$ with unique measure $\mu$ of maximal entropy. The \emph{variance} $\Var_X(F)$ of a measurable function $F: M \longrightarrow \R$ is given by
\[
\Var_X(F) = \int_{-\infty}^\infty \Big(\int_M F(F \circ \phi^{t}_{X})\mu - \big(\int_M F \mu \big)^2\Big) dt.
\]
\end{definition}

\noindent We are now ready to state the derivative formulas of Pollicott \cite{pollicott1994derivatives}. This result will be the key tool in the proof of Proposition \ref{prop:main_entropy_prop}. 

\begin{thm} \label{thm:pollicott_derivatives} \cite{pollicott1994derivatives} Let $X_s$ be a smooth family of Anosov vector fields on a closed manifold $M$ for $s \in (-\varepsilon,\varepsilon)$. Let $F_s$ be the corresponding family of functions such that $F_s \cdot X_0 = {\Psi_s^*}X_s $. Expand $F_s$ as
\[
F_s = 1 + f \cdot s + g \cdot s^2 + O(s^3),
\]
and consider the smooth function $\TopEn:(-\varepsilon,\varepsilon) \longrightarrow \R$ given by $s \longmapsto \TopEn(X_s)$. Then we have
\[
\frac{d}{ds}(\TopEn(X))|_{s = 0} = \TopEn(X_0) \cdot \int_M f \mu_{BM}, \qquad \text{and}
\]
\[
\frac{d^2}{ds^2}(\TopEn(X))|_{s = 0} = \TopEn(X_0) \cdot  \Big(\Var_X(f) + \big(\int_M f \mu_{BM} \big)^2 + 2 \int_M g \mu_{BM} - \int_M f^2 \mu_{BM} \Big).
\]\end{thm}

We conclude this part of the section by noting that the Bowen-Margulis measure is particularly nice in the setting that we will need. We formalize this with the following lemma.

\begin{lemma} \label{lem:BM_is_regular} Let $M$ be a closed manifold, and  $X$ be a smooth volume-preserving Anosov vector field on $M$ without a global cross-section. Then the Bowen-Margulis measure $\mu_{BM}$ is regular. 
\end{lemma}

\begin{proof} Since $X$ is volume preserving, the non-wandering set of $X$ is $M$. Thus $X$ is transitive \cite[Thm 5.3.52]{fh2019hyperbolic}. By the Plante alternative \cite[Thm 8.1.3]{fh2019hyperbolic} and the assumption that $X$ admits no global cross-section, the strong-stable and strong-unstable leaves of $X$ are dense in $M$. An Anosov flow with dense strong-stable and strong-unstable leaves admits an alternate construction of the Bowen-Margulis measure due to Hamenstadt \cite[Thm 7.5.2]{fh2019hyperbolic} which shows that it is a Borel measure. Thus the measure is regular since any Borel measure on a manifold is regular. \end{proof}

\subsection{The topological entropy map} We conclude this section with a proof of Proposition \ref{prop:main_entropy_prop} using the tools discussed above. We start by fixing
\[
\text{a closed $3$-manifold $M$} \qquad\text{and}\qquad \text{a volume form $\mu$}.
\]
The topological entropy yields a map
\[
\TopEn:\mathfrak{X}_\mu(h) \longrightarrow \R
\]
on the subset of differentiable $\mu$-preserving vector fields on $M$ with helicity $h$. This map is continuous on the space of Anosov vector fields with respect to the $C^1$-topology. Our final goal in this section is to prove the following result.

\begin{prop} \label{prop:main_entropy_prop} Let $\mathcal{U} \subset \mathfrak{X}_\mu(h)$ be an open set of Anosov vector fields and suppose that $M$ is not the mapping torus of an Anosov diffeomorphism. Then
\[
\TopEn|_{\mathcal{U}} \quad\text{is non-constant.}
\]
\end{prop}

\begin{proof}  We construct a smooth family of vector fields $X_\varepsilon$ in $\mathcal{U}$ whose topological entropy cover an open set in $\R$. Fix a vector field $X \in \mathcal{U}$. We may assume without loss of generality that
\[
X \text{ is smooth} \qquad\text{and}\qquad \mu(M) = 1.
\]
Here $\mu(M)$ is the volume of $M$ with respect to $\mu$. We will regard $\mu$ as a measure in this proof.

\vspace{3pt}

{\bf Step 1: Family  of vector fields} We start by describing the family of vector fields that we will use. Choose a smooth function
\begin{equation} \label{eq:perturbation_function}
f:M \longrightarrow \R \qquad\text{such that}\qquad \int_M f \mu = 0.
\end{equation}
Let $W_\varepsilon$ and $\mu_\varepsilon$ denote the family of vector fields and volume forms given by
\[
W_\varepsilon = (1 - \varepsilon f)^{-1} \cdot X \qquad\text{and}\qquad \mu_\varepsilon = (1 - \varepsilon f) \cdot \mu,
\]
for $\varepsilon>0$ small.
Note that $\mu_\varepsilon(M) = \mu(M)$ for all $\varepsilon$ by our choice of $f$. Therefore, by Moser stability, there is a smooth family of diffeomorphisms 
$${\Psi_\varepsilon}:M \longrightarrow M,$$ 
satisfying $\Psi_\varepsilon^*\mu_\varepsilon = \mu$. We denote the vector fields obtained by pulling back by $\Psi_\varepsilon$ by
\[
X_\varepsilon = \Psi^*_\varepsilon W_\varepsilon
\]
Note that the interior product $\iota_{X_\varepsilon}\mu$ is given by $\Psi^*_\varepsilon(\iota_X\mu)$. It follows that the helicity of $X_\varepsilon$ is constant and equal to $h$. By Theorem \ref{thm:pollicott_derivatives}, the derivatives of the entropy are given by
\begin{equation} \label{eq:main_entropy_prop1}
\frac{d}{d\varepsilon}(\TopEn(X_\varepsilon)) = \TopEn(X) \cdot \int_M f \mu_{BM},
\end{equation}
\begin{equation} \label{eq:main_entropy_prop2}
\frac{d^2}{d\varepsilon^2}(\TopEn(X_\varepsilon)) = \TopEn(X) \cdot \Var_X(f).
\end{equation}
Here $\mu_{BM}$ is the Bowen-Margulis measure (Definition \ref{def:BM_measure}), which is a regular probability measure in this setting by Lemma \ref{lem:BM_is_regular}. We break the remaining analysis into two cases depending on whether or not this measure agrees with the given smooth volume form as a measure.

\vspace{3pt}

{\bf Step 2: Inequality Case.} In the first case, we assume that the Bowen-Margulis measure $\mu_{BM}$ is not equal to the preserved volume form $\mu$ as a measure. We need the following lemma.

\begin{lemma} \label{lem:distinguishing_functions} Let $\mu$ and $\nu$ be distinct regular probability measures on $M$. Then there is a smooth function $f:M \longrightarrow \R$ such that
\[
\int_M f \mu = 0 \qquad\text{and}\qquad \int_M f \nu = 1. 
\]
\end{lemma}

\begin{proof} Let $A \subset M$ be a subset such that $\mu(A) \neq \nu(A)$. Since $M$ is compact and $\mu,\nu$ are regular, we may find a sequence of compact sets $K_i$ and open sets $U_i$ where $K_i \subset A \subset U_i$ and as $i\rightarrow \infty$ we have
\[
\mu(U_i \setminus K_i) \longrightarrow 0 \qquad\text{and}\qquad \nu(U_i \setminus K_i) \longrightarrow 0.
\]
Fix a sequence of smooth functions $\chi_i:X \longrightarrow [0,1]$ that are $1$ on $K_i$ and $0$ on $X \setminus U_i$. It follows that
\[
\int_M \chi_i \mu \longrightarrow \mu(A) \quad\text{and}\quad \int_M \chi_i \nu \longrightarrow \nu(A),
\]
and thus for large $i$ we have
$$\int_M \chi_i(\mu - \nu) \neq 0.$$
Let $g$ be $\chi_i$ for sufficiently large $i$, and let $A$ and $B$ be the integral of $g$ with respect to $\mu$ and $\nu$ respectively. It is simple to check the desired result for the function $f = (B - A)^{-1} \cdot (g - A)$.
\end{proof}
\noindent We return to the construction of the family $X_\varepsilon$ in Proposition \ref{prop:main_entropy_prop}. Applying Lemma \ref{lem:distinguishing_functions} to the probability measures $\mu$ and $\mu_{BM}$, we find a function function $f$ satisfying
\[
\int_M f \mu_{BM} = 1, \quad \int_M f \mu=0.
\]
We use this $f$ in \eqref{eq:perturbation_function}, and then by the derivative formula (\ref{eq:main_entropy_prop1}), $0$ is a regular value of the map $\varepsilon \longmapsto \TopEn(X_\varepsilon)$. It follows that the map is non-constant in a neighborhood of $X$ inside $\mathcal{U}$.

\vspace{3pt}

{\bf Step 3: Equality Case.} In the second case, we assume that $\mu = \mu_{BM}$. Fix a closed orbit $\Gamma \subset M$ of $X$ and choose a smooth function $f:M \longrightarrow \R$ so that
\[
\int_M f \mu = \int_M f \mu_{BM} = 0 \qquad\text{and}\qquad f|_\Gamma = 1.
\]
We claim that $f$ has non-zero variance with respect to $X$. Indeed, by Pollicott \cite[p. 465]{pollicott1994derivatives} the variance is non-negative, and it is zero if and only if there is a continuous function $u:M \longrightarrow \R$ such that
\[
\int_0^t f \circ \phi^{s}_{X}(x) ds = t \int_M f \mu_{BM} + u \circ \phi^{t}_{X}(x) - u(x)
\qquad\text{for all $t$}.\]
Here $\phi^{t}_{X}$ is the flow of the vector field $X$. Since $f$ has zero integral with respect to $\mu$ by (\ref{eq:perturbation_function}) and $\mu = \mu_{BM}$ in this case, this simplifies to the cocycle equation
\[
\int_0^t f \circ \phi^{s}_{X}(x) ds =  u \circ \phi^{t}_{X}(x) - u(x) \qquad\text{for all }t\in \mathbb{R} \text{ and }x\in M.
\]
In particular, if $x$ lies on a periodic orbit of period $T$ then one must have
\[
\int_0^T f \circ \phi^{s}_{X}(x) ds = 0.
\]
But $f$ has non-zero integral along $\Gamma$. It follows that $\Var_X(f) > 0$. Now we consider the map $\varepsilon \longmapsto \TopEn(X_\varepsilon)$. By the smoothness of this map (Theorem \ref{thm:smoothness_of_entropy}), combined with the derivative formulas (\ref{eq:main_entropy_prop1}) and (\ref{eq:main_entropy_prop2}), we have the following expansion of the topological entropy
\[
\TopEn(X_\varepsilon) = \TopEn(X)(1 + \Var_X(f) \cdot \varepsilon^2) + O(\varepsilon^3).
\]
This implies that $\TopEn$ is non-constant in a neighborhood of $X$ in $\mathcal{U}$. This finishes this case and the proof of Proposition \ref{prop:main_entropy_prop}. \end{proof}

\section{Nowhere density of coadjoint orbits}\label{sec:orbits}
In this last section, we establish the everywhere non-density of coadjoint orbits on any closed three-manifold, as stated in Theorem \ref{thm:main2}. The strategy is to show that there is a dense set of coadjoint orbits where each orbit satisfies the following property: there is a neighborhood $U$ of it inside $\mathfrak{X}_\mu^0 (h)$ and a coadjoint invariant functional defined on $U$ which is continuous and everywhere non-constant. Such functionals are quantities that can be defined using ``robust" properties of the zeroes or of the periodic orbits of certain vector fields. 

\subsection{Definition of the local functionals}\label{ss:localfunc}
In this section, we define functionals in the neighborhood of certain coadjoint orbits inside their helicity level set. These functionals measure dynamical invariants of either zeroes or periodic orbits of vector fields, as suggested in \cite[Remark 9.4 (B)]{ak}.
\medskip
\paragraph{\textbf{Local invariants via hyperbolic zeroes.}} First, we will define a local functional in the neighborhood of a coadjoint orbit of a vector field that has only hyperbolic zeroes, and at least one. Such a type of functionals were considered by Khesin, see \cite[Remark 3]{Kh2024}. We first define the subset
$$\mathcal{HZ}(h)= \{ X\in \mathfrak{X}_\mu(h) \mid X \text{ only has hyperbolic zeros}\}.$$
We denote by $\mathcal{HZ}_1(h)$ the set of those vector fields in $\mathcal{HZ}(h)$ that have at least one zero. Let $\mathcal{O}(W)$ be the coadjoint orbit of a vector field $W \in \mathcal{HZ}_1(h)$. There is a small enough $C^{1}$ neighborhood $\mathcal{V}_W\subset \operatorname{HZ}(h)$ of $\mathcal{O}(W)$ such that all vector fields in $\mathcal{V}_W$ have the same number of zeros as $W$, and these zeros vary continuously by the implicit function theorem. It is thus natural to try to construct an invariant near the coadjoint orbit using local invariants of these zeroes; this is done, for instance, in \cite{Kh2024}. For a given vector field $Y\in \mathcal{V}_W$, we introduce the notation
$$ \operatorname{Z}(Y)= \{p\in M \mid Y(p)=0\}, $$
and for each $p\in \mathcal{Z}(Y)$
$$\operatorname{M}(p)= \{\lambda \in \mathbb{C} \mid \lambda \text{ is an eigenvalue of } DY(p)\}.$$
Using these, we define the functional
\begin{align*}
    \mathcal{S}_W: \mathcal{V}_W &\longrightarrow \mathbb{R}\\
    Y &\longmapsto \sum_{\substack{p\in \operatorname{Z}(Y)}} \Big(\sum_{\lambda \in \operatorname{M}(p)}{{\operatorname{Re}(\lambda)^2}} \Big).
\end{align*}
Notice that this functional is continuous in the $C^{1}$-topology, and $\mathcal{S}_W(X) >0$ whenever $X\in \mathcal{V}_W$ has at least one hyperbolic zero. It has the same value for any two vector fields in $\mathcal{V}$ that lie in the same coadjoint orbit, since the eigenvalues of $DY$ at a zero are preserved when taking the pushforward of $Y$ by any diffeomorphism.
\medskip

\paragraph{\textbf{Local invariants via short non-degenerate orbits}}
We will now define a local functional in the neighborhood of certain non-vanishing vector fields. First, we denote by 
$$\mathcal{P}(n,h) \subset \mathfrak{X}_\mu^0(h),$$ 
the set of non-vanishing vector fields of helicity $h$ satisfying that every periodic orbit of period $\leq n$ is non-degenerate, and that they admit at least one periodic orbit of period $<n$. The first lemma establishes a control on the periodic orbits of small period for perturbations of $Y$.

\begin{lemma}\label{lem:shortperiodic}
Let $Y$ be a vector field in $\mathcal{P}(n,h)$, and $\kappa>0$ small enough such that there is no periodic orbit of period $n-\kappa$, but there is at least one of period smaller than $n-\kappa$. Let $\gamma_1,...,\gamma_r$ be the periodic orbits of $Y$ of period $<n-\kappa$. Then there is a neighborhood $A_Y$ of $Y$ in $\mathfrak{X}_\mu(M)$ such that any vector field $Z$ in that neighborhood admits exactly $r$ orbits of period $< n-\kappa$, given by deformations of $\gamma_1,...,\gamma_r$ obtained by the implicit function theorem.
\end{lemma}
\begin{proof}
We first claim that for each $\gamma_i$, there are tubular neighborhoods $V_i$ of $\gamma_i$ and some $\delta_i$ with the following property. If $Z$ satisfies 
$$ \norm{Z-Y}_{C^1}<\delta_i, $$ 
then it admits exactly one closed orbit intersecting $V_i$. The latter is given by a perturbation of $\gamma_i$ obtained by the implicit function theorem. 

This is easily argued as follows (the argument is as in \cite[Lemma 2]{Pei} for the hyperbolic case). Given any $k\in \mathbb{N}$, there are small enough transverse disks $D_0\subset D_1$ centered at a point in $\gamma_i$ for which the first-return map
$$\varphi: D_0 \longrightarrow D_1,$$
can be iterated $k$ times, and $p$ is the only fixed point of $\varphi^k$. In addition, one easily checks that the non-degeneracy implies that the graph
\begin{align}
    F: D_0 &\longrightarrow D_1\times D_1 \\
     p &\longmapsto (p, \varphi^k(p))
\end{align}
is transverse to the diagonal $\Delta=\{(q,q)\mid q\in D_1\}$ along $(p,p)$. Then, in a tubular neighborhood $V_i$ of $\gamma_i$, a vector field $Z$ close enough to $Y$ will satisfy that the first-return map and its iterates up to $k$ still have a unique fixed point in $D_0$. Now, if $T>0$ denotes a lower bound on the first-return time of the flow along the transverse section, we choose $k>0$ such that $kT>n$. Then any new periodic  of $Z$ near $\gamma_i$ will have a period $>kT$, and thus strictly greater than $n$, as we wanted.

Let $\eta_1,...,\eta_s$ be the periodic orbits of $Y$ with period in $(n-\kappa,n]$. Arguing analogously, there are tubular neighborhoods $U_j$ of $\eta_j$ and some $\delta_j'$ such that any vector field that is $\delta_j'$-close to $Y$ in the $C^1$-topology admits at most\footnote{An orbit of period exactly $n$ could become of larger period.} one orbit of period $\leq n$. Taking $\delta_j'$ small enough, we can assume as well that the period of these orbits is larger than $n-\kappa$.\\

We can now cover $M\setminus \left(V_1\cup ... V_r \cup U_1 \cup ... U_s\right)$ by finitely many neighborhoods $A_1,...,A_l$ such that for each $A_k$ there is a $\tilde \delta_k$ satisfying that if $Z$ is such that
$$\norm{Z-Y}_{C^1}<\tilde \delta_k,$$
then there is no periodic orbit of period {$\leq n$} intersecting $A_k$. Taking 
$$\delta:=\min{\{\delta_i\}_{i=1}^r, \{\delta_j'\}_{j=1}^r, \{\tilde \delta_k\}_{k=1}^l },$$
the neighborhood
$$\mathcal{V}=\{Z\in \mathfrak{X}(M)\mid \norm{Z-Y}_{C^1}<\delta \}$$
satisfies the required property.
\end{proof}

For each $Y\in \mathcal{O}(W)$, let $V_Y$ be the open set $A_Y$ given by Lemma \ref{lem:shortperiodic} intersected with $\mathfrak{X}_\mu(h)$, and define the open set
$$\mathcal{V}_W= \bigcup_{Y\in \mathcal{O}(W)} V_Y.$$
Notice that this is a $C^1$-open neighborhood of $\mathcal{O}(W)$. Given a periodic orbit $\gamma$ of a vector field, we define by $\tau(\gamma)$ its primitive period. Define the functional
\begin{align}
    P_{min}^W: \mathcal{V}_W &\longrightarrow \mathbb{R}\\
        X&\longmapsto \min{ \tau(\gamma) \mid \gamma \text{ periodic orbit of } X }
\end{align}
That this functional is constant along coadjoint orbits follows directly from its definition. Indeed, the periods of the periodic orbits of a flow are invariant under conjugation by any diffeomorphism, thus in particular by volume-preserving diffeomorphisms. From our previous lemma, we can deduce its continuity. 
\begin{cor}
The functional $P_{min}^W$ is continuous, where we endow $\mathcal{V}_W$ with the $C^1$-topology.    
\end{cor}
\begin{proof}
    By Lemma \ref{lem:shortperiodic}, the shortest period of a periodic orbit must be smaller than $n-\kappa$, and it is given by the minimum among the periods of the deformations of the orbits $\gamma_1,...,\gamma_r$. But the periods of these orbits vary continuously with respect to the $C^1$-topology, from which the statement follows.
\end{proof}

\subsection{Everywhere non-triviality of local functionals}
Our goal now is to show that the previously defined functionals are everywhere non-constant.  The idea is simple: any constant rescaling $Y\rightarrow (1+\tau) Y$ will change the value of any of the two functionals; however, it will also change the value of the helicity. By a local helicity correction as in Section \ref{ss:helcorrec}, we can bring back the helicity of $(1+\tau)Y$ to $h$ without changing the vector field near the zeroes or relevant periodic orbits. If the perturbation is small enough, that helicity correction will not change the value of the functional, showing that the latter is non-constant everywhere.

\begin{prop}\label{prop:nonconstantlocal}
    Let $W$ be a vector field in $\mathcal{HZ}_1(h)$ (in $\mathcal{P}(n,h)$, respectively). Then $\mathcal{S}_W$ (or $P_{min}^W$, respectively) is everywhere non-constant. 
\end{prop}
\begin{proof}
The argument is very similar in both cases, so we will argue simultaneously and specify $W\in \mathcal{HZ}_1(h)$ or $W\in \mathcal{P}(n,h)$ to distinguish between the two cases. We fix a $W$ in one of the two sets, and recall that in the latter case, the domain of the functional is contained in the set of vector fields all whose periodic orbits of period $\leq n$ are non-degenerate.\\

    Let $\mathcal{U}$ be an open set in the domain of the suitable functional.
    Choose some vector field $X\in \mathcal{U}$, and some open neighborhood $V\subset M$ satisfying that no zero of $X$ lies in $V$, if $W\in \mathcal{HZ}_1(h)$, or no periodic orbit of period $\leq n$ intersects $V$ if $W\in \mathcal{P}(n,h)$. Fix some $\hat \delta$ such that any vector field $Y$ such that $\norm{Y-X}_{C^1}<\hat \delta$ we have
    \begin{itemize}
        \item[-] If $W\in \mathcal{HZ}_1(h)$, then we want that $Y$ satisfies that all zeroes of $Y$ are hyperbolic, given by perturbations of those of $X$, and that they do not intersect $V$,
        \item[-] If $W\in \mathcal{HZ}_1(h)$, we require that the periodic orbits of small period of $Y$ satisfy Lemma \ref{lem:shortperiodic}, and that these periodic orbits do not intersect $V$,
        \item[-] if $\mathcal{H}(Y)=h$, then $Y\in \mathcal{U}$.
    \end{itemize}
    The latter condition holds for a small enough $\hat \delta$ due to the openness of $\mathcal{U}$. Consider a reparametrization of $X$ of the form $\hat X= \sqrt{1+\eta}X$, where $\eta>0$ is such that
    \begin{equation}\label{eq:zeroes2}
        \norm{\hat X-X}_{C^1}<\delta(\eta)<\hat \delta,
    \end{equation}
    where $\delta$ goes to zero with $\eta$.
    Notice that 
    $$\mathcal{S}_W(\hat X)> \mathcal{S}_W(X), \enspace \text{if} \enspace  W\in \mathcal{HZ}_1(h)$$ 
    for any $\eta>0$. Similarly, we have
    $$P_{min}^W(\hat X)> \mathcal{S}(X), \enspace \text{if} \enspace  W\in \mathcal{P}(n,h)$$
    for any $\eta>0$. On the other hand, the helicity of $\hat X$ is 
    $$\mathcal{H}(\hat X)= (1+\eta)\mathcal{H}(X),$$
and thus $\hat X$ does not belong to $\mathcal{HZ}(h)$ anymore unless $h=0$. By Corollary \ref{cor:localmod}, applied to the open set $V$, for any $\delta_1>0$ small enough, there is $\delta_2$ (going to zero with $\delta_1$) such that given any exact vector field $Y$ satisfying 
$$\norm{Y-X}_{C^\infty}<\delta_1,$$
there is a vector field $Z$ such that
\begin{itemize}
    \item[-] $\norm{Z-X}<\delta_2$,
    \item[-] $\mathcal{H}(Z)=h$,
    \item[-] $Z|_{M\setminus V}=Y$.
\end{itemize}
We choose $\delta_1$ small enough so that $\delta_2<\hat \delta$. Taking $\eta$ small enough, we have
$$\norm{{\hat X}-X}<\delta_1.$$
We find then a vector field $Z$ such that
$$Z|_{M\setminus V}=\hat X,\qquad \text{and} \qquad \norm{X-Z}_{C^\infty}<\hat \delta.$$
In particular, it lies in $\mathcal{U}$ and since all the zeroes (or short periodic orbits) of $Z$ lie in $M\setminus V$, we have 
$$\mathcal{S}_W(Z)= \mathcal{S}_W({\hat X}) \neq \mathcal{S}_W(X), \enspace \text{if} \enspace  W\in \mathcal{HZ}_1(h),$$
and 
$$P_{min}^W(Z)= P_{min}^W({\hat X}) \neq P_{min}^W(X), \enspace \text{if} \enspace  W\in \mathcal{P}(n,h),$$
as we wanted to show.
\end{proof}

\subsection{Proof of Theorem \ref{thm:main2}} \label{ss:proofmain2}
In this section, we use the previously introduced functionals to prove Theorem \ref{thm:main2}.
\medskip

\textbf{Step 1: Restriction to a subset of $\mathfrak{X}_\mu(h)$.}  Given a vector field $X\in \mathfrak{X}_\mu(h)$, recall that we denote by $\mathcal{O}(X)\subset \mathfrak{X}_\mu(h)$ its coadjoint orbit. Suppose that $\mathcal{U}\subset \mathfrak{X}_\mu(h)$ is a subset such that 
\begin{itemize}
    \item[-] it is open and dense (in the $C^1$-topology),
    \item[-] $u\in \mathcal{U}$ implies that $\mathcal{O}(u)\subset \mathcal{U}$.
\end{itemize}
In such a situation, to prove that no vector field in $\mathfrak{X}_\mu(M)$ has a somewhere dense coadjoint orbit in its helicity level set, it is enough to prove it for vector fields in $\mathcal{U}$. Indeed, consider a vector field $X \in \mathfrak{X}_\mu(h)$ and assume that its coadjoint orbit $\mathcal{O}(X)$ is somewhere $C^1$-dense. Then $\mathcal{U}\cap \mathcal{O}(X)\not = \emptyset$, and thus the vector field $X$ belongs to $\mathcal{U}$. \\

For each $n$, we define the open set
$$ \mathcal{U}_n=\bigcup_{Y\in \mathcal{P}(n,h)} A_Y\cap \mathfrak{X}_\mu^0(h) $$
The set $\mathcal{U}$ that we will use is then defined by
$$\mathcal{U}:=\mathcal{HZ}_1(h)\cup \left(\bigcup_{n\geq 1} \mathcal{U}_n\right) \subset \mathfrak{X}_\mu^0(h).$$
Notice that $\mathcal{HZ}_1(h)$ contains only vector fields with zeroes, while $\bigcup_{n\geq 1} U_n$ contains only non-vanishing vector fields.
\medskip

\textbf{Step 2: $\mathcal{U}$ is open and dense.} Notice that $\mathcal{U}$ is open by definition, so we are left with proving that it is dense. In the following, we present an argument that works for any possible value of $h$, including the harder case of $h=0$ (which the reader should keep in mind in the following proof). When $h\neq 0$, some arguments can be simplified by avoiding the use of results proven in Section \ref{ss:helcorrec}. Indeed, one can use density statements known in $\mathfrak{X}_\mu^0(M)$ and then rescale by a suitable constant to find a vector field of helicity $h\neq 0$ with the required properties. 
\medskip

 {We start by showing that $\mathcal{U}$ is dense among vector fields with zeroes.} Let $X$ be a non-trivial arbitrary vector field of given helicity, and fix some open set $V\subset M$ where $X$ does not vanish. By Corollary \ref{cor:localmod}, given $\delta_1$, there is a $\delta_2$ (going to zero with $\delta_1$) such that for any exact vector field $Y$ such that $\norm{X-Y}<\delta_1$, there is a vector field $Z$ with the same helicity as $h$ and such that
 $$\norm{X-Z}_{C^1}<\delta_2, \qquad \text{and} \qquad Z_{M\setminus V}=Y.$$
Given any $\delta>0$, we choose $\delta_1$ small enough so that $\delta_2<\delta$. By the density of vector fields with hyperbolic zeroes among exact vector fields, we can find some $Y$ that is $\delta_1$-close to $X$ in the $C^1$-topology, that has only hyperbolic zeroes. Then the vector field $Z$ given by Corollary \ref{cor:localmod} has only hyperbolic zeroes and has the same helicity as $X$, thus either $Z \in \mathcal{U}$ or it is non-vanishing (in which case, by the discussion below it can further be perturbed to be in $\mathcal{U}$), and it is $\delta$-close to $X$. 
\medskip

To conclude, we want to show that $\mathcal{U}$ is dense among non-vanishing vector fields in $\mathfrak{X}_\mu^0(h)$.
\begin{prop}
The set $\mathcal{U}$ is dense among non-vanishing fields in $\mathfrak{X}_\mu^0(h)$.
\end{prop}
\begin{proof}
    By its definition, it is enough to show that $\bigcup_{n\geq 1} \mathcal{P}(n,h)$ is dense. Given a non-vanishing vector field $X\in \mathfrak{X}_\mu^0(h)$, by Theorem \ref{thm:closingh} we can perturb it in the $C^1$-topology to assume that it has at least one periodic orbit. We still denote this vector field by $X$ and denote by $\gamma$ the periodic orbit. Fix an open set $V$ in the complement of $\gamma$, and a small enough neighborhood $W$ (disjoint from $V$) of some point $p$ in $\gamma$. It is standard (for instance using \cite{Rob2}[Lemma 19]) that for any $\delta>0$, there exists an exact vector field $X'$ satisfying: 
    \begin{itemize}
        \item[-] $\norm{X-X'}_{C^\infty}<\delta$,
        \item[-] $X'|_{M\setminus W}=X$,
        \item[-] $\gamma$ is a non-degenerate periodic orbit of $X'$.
    \end{itemize} 
    Using Corollary \ref{cor:localmod} in $V$, this perturbation can be achieved while keeping the helicity fixed. We find thus a vector field $Y$, arbitrarily close to $X$, with helicity $h$ and a non-degenerate periodic orbit $\gamma$. Since it is non-degenerate, there is an open neighborhood $\mathcal{V}\subset \mathfrak{X}_\mu^0(h)$ of $Y$ where every vector field admits a deformation of $\gamma$ as a periodic orbit, by the implicit function theorem. Thus, using Theorem \ref{thm:kupkah}, we can find a vector field $Z$ arbitrarily $C^\infty$-close to $Y$ that is non-degenerate, has helicity $h$, and admits at least one periodic orbit. Thus, choosing $n$ large enough, we have that ${Z}\in \mathcal{P}(n,h)$ as we wanted to prove.
\end{proof}

\paragraph{\textbf{Step 3: the statement holds in $\mathcal{U}$}.} The final step is to prove that the coadjoint orbit of any vector field in $\mathcal{U}$ is nowhere dense in it. Given $X\in \mathcal{U}$, Section \ref{ss:localfunc} shows that in any small enough neighborhood of a vector field $Y\in \mathcal{O}(X)$ inside $\mathfrak{X}_\mu^0(h)$ there is a functional, constant along coadjoint orbits, that is continuous in the $C^1$-topology. Proposition \ref{prop:nonconstantlocal} shows that this functional is non-constant on such neighborhoods. But notice that if $\mathcal{O}(X)$ was somewhere dense in $\mathfrak{X}_\mu^0(h)$, it would be dense in any small enough neighborhood of some element $Y\in \mathcal{O}(X)$, which is precluded by the {non-constancy} of the functional there. This concludes the proof of Theorem \ref{thm:main2}.

\appendix

\section{A Franks-type lemma for conservative flows}\label{app}

The goal of this appendix is to prove a Franks-type lemma (Lemma \ref{lem:franks}) for conservative vector fields, which is of independent interest, and which fills a gap in several results in the literature on the generic behavior of 3D conservative and 4D Hamiltonian systems. Let us recall the statement, following the notation in Section \ref{ss:hellipticorbits}.

\begin{lemma}[Frank's-type Lemma for volume-preserving flows]\label{lem:franks2}
Fix any $\varepsilon>0$ and a $C^1$ map $\mathfrak{c}: [0, T] \to \mathfrak{sl}(2, \RR)$, compactly supported in $[0, T]$, and satisfying
\begin{equation}\label{psize}
\sup_{t \in [0, T]} ||\mathfrak{a}_{X}(t)-\mathfrak{c}(t)|| < \varepsilon.
\end{equation}
Then there is a constant $K$, depending only $||X||_{C^{0}}$ and $C_{\tau}$, such that the following holds: given any arbitrarly small tubular neighborhood of $\Gamma$, there is a vector field $Y$ in $\mathfrak{X}^{0}_{\mu}(M)$ with the following properties:
\begin{enumerate}
\item $X=Y$ outside the tubular neighborhood. 
\item $||X-Y||_{C^{1}(M)} \leq K \varepsilon$. 
\item For any $t\in [0, T]$, $\phi^{t}_{X}(p)=\phi^{t}_{Y}(p)$.
\item For any $t\in [0, T]$ we have
\[
\bigg(\frac{d}{dt} A_{Y}(t)\bigg) \circ A_{Y}^{-1}(t)=\mathfrak{c}(t),
\]
where
\[
A_{Y}(t):=\tau_{\phi^{t}_{Y}(p)} \circ D_{p} \phi^{t}_{Y} \circ \tau^{-1}_{p}.
\]

  \end{enumerate}
\end{lemma}

\begin{proof}
The proof will be divided in three steps. In step 1 we construct a compactly supported vector field in $\RR^2 \times [0, T]$, that will serve as the basic building block of the perturbation. Step 2 ensures the existence of a flow-box chart around the orbit $\Gamma$ with suitably controlled derivatives. In the last step we use this flox-box chart to embed the local construction in step 1 into the manifold $M$, and show that the resulting vector field satisfies the claims in the Lemma. 
\\
\paragraph{\textbf{Step 1: Local construction}}

Consider a smooth map $\mathfrak{b}: \RR \rightarrow \mathfrak{sl}(2, \RR)$ with compact support inside $[0, T]$.

\begin{prop}\label{prop:localW}
Given $\kappa>0$ as small as desired, there is a smooth vector field $\tilde{W}_{\kappa}$ on $\RR^2 \times [0, T]$ such that
\begin{enumerate}

\item $\tilde{W}_{\kappa}$ is supported on $\{(x, y, t) \in \RR^3:\,x^{2}+y^2 < \kappa^2,\,\, t\in (0, T) \}$.

\item $\tilde{W}_{\kappa}$ preserves any volume form of the type $\rho(t) dx \wedge dy \wedge dt$, for $\rho(t)$ any positive smooth function, and it is exact. 

\item We have 
\[
D_{(0, 0, t)} \tilde{W}_{\kappa}=\begin{bmatrix} b_{11}(t) & b_{12}(t) & 0 \\ b_{21} (t) & b_{22}(t) & 0\\ 0& 0 &0 \\ \end{bmatrix}
\]
where $b_{ij}(t)$ are the entries of the matrix $\mathfrak{b}(t)$.

\item We have the bounds:
\begin{equation}\label{eq:W0}
||\tilde{W}_{\kappa}||_{C^{0}} \leq 2 \kappa ||\mathfrak{b}(t)||_{C^{0}},
\end{equation}
\begin{equation}\label{eq:Wt}
||\partial_{t}\tilde{W}_{\kappa}||_{C^{0}}\leq 2 \kappa ||\mathfrak{b}'(t)||_{C^{0}} ,
\end{equation}
\begin{equation}\label{eq:Wxy}
||\partial_{x}\tilde{W}_{\kappa}||_{C^{0}}+||\partial_{y}\tilde{W}_{\kappa}||_{C^{0}} \leq 1000 ||\mathfrak{b}(t)||_{C^{0}}.
\end{equation}

\end{enumerate}
\end{prop}
The remaining of this Step 1 section proves this Proposition.
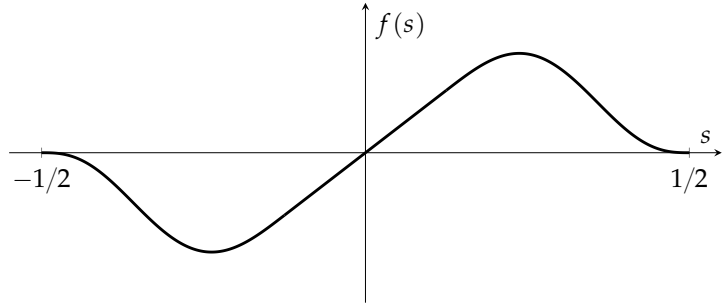
\begin{proof}

\begin{figure}[h]
\centering
\begin{tikzpicture}[scale=0.90]
\begin{axis}[
  axis lines=middle,
  xlabel={$s$}, ylabel={$f(s)$},
  xmin=-2.2, xmax=2.2,
  ymin=-1.2, ymax=1.2,
  xtick={-2,2},
  xticklabels={$-1/2$,$1/2$},
  ytick=\empty,
  samples=300,
  width=12cm,height=6cm
]

\addplot[black,very thick,domain=-0.5:0.5,samples=2] {x};

\addplot[black,very thick,domain=0.5:2,samples=200] 
  { x * (1 - ( 6*((x-0.5)/1.5)^5 - 15*((x-0.5)/1.5)^4 + 10*((x-0.5)/1.5)^3 ) ) };

\addplot[black,very thick,domain=-2:-0.5,samples=200] 
  { x * (1 - ( 6*((-x-0.5)/1.5)^5 - 15*((-x-0.5)/1.5)^4 + 10*((-x-0.5)/1.5)^3 ) ) };

\end{axis}
\end{tikzpicture}

\end{figure}

Let $f: \RR \to \RR$ be a smooth function whose graph is as sketched in the figure below.
More precisely, we will choose $f$ so that $f(s)=s$ on a neighborhood of the origin, it is zero outside $[-1/2, 1/2]$, and $||f||_{C^{1}}\leq 1$, $||f||_{C^2} \leq 100$. As an example of such a function, start with $\varphi(s):=e^{-1/s}$ for $s\geq 0$ and $\varphi(s):=0$ for $s<0$, define 
\[
h(s):=\varphi\bigg(\frac{s-1/32}{\varphi(1/2-s)}\bigg)
\]
for $s \in (1/32, 1/2)$, with extension by $0$ for $s\leq 1/32$ and $1$ for $s\geq 1/2$, and set
\[
f(s):=s(1-h(s)-h(-s)).
\]
Let us define the rescaled function
\[
f_{\kappa}(s):= \kappa f(s/ \kappa).
\]
The function $f_{\kappa}$ has support inside $[- \kappa/2, \kappa/2]$, and satisfies $|f_{\kappa}(s)| \leq \kappa$, $|f'_{\kappa}(s)| \leq 1$ and $|f''_{\kappa}(s)| \leq 100/\kappa$.

Now define the function
\[
\psi(x, y, t)=-b_{11}(t) f_{\kappa}(x) f_{\kappa}(y)+\frac{b_{21}(t)}{2} f_{\kappa}(x)f_{\kappa}(x)+\frac{b_{12}(t)}{2} f_{\kappa}(y) f_{\kappa}(y),
\]
and the associated vector field (gradient of $\psi$ rotated ninety degrees in the $x,y$ plane)
\[
\tilde{W}_{\kappa}(x, y, t)=(-\partial_{y} \psi) \frac{\partial}{\partial x}+ (\partial_{x} \psi) \frac{\partial}{\partial y}.
\]
By construction, the vector field $\tilde{W}_{\kappa}$ is supported in $\{x^2+y^2 \leq \kappa^2\} \times (0, T)$, and is divergence free with respect to any volume form $\rho(t) dx \wedge dy \wedge dt$. It is exact because it has compact support on a simply connected open set, an explicit primitive is $-\rho(t)\psi(x, y, t) dt$. Thus we get the first two claims (i) and (ii). 

Moreover, notice that for $(x, y)$ in a neighbourhood of the origin we have
\[
\psi(x, y, t)=-b_{11}(t) xy+\frac{b_{21}(t)}{2} x^2+\frac{b_{12}(t)}{2} y^2, 
\]
and
\[
\tilde{W}_{\kappa}=(b_{11}(t) x+b_{12}(t) y) \partial_{x}+(-b_{11}(t) y+b_{21}(t) x) \partial y.
\]
In particular, since $b_{22}(t)=-b_{11}(t)$, 
\[
D_{(0, 0, t)} \tilde{W}_{\kappa}=\begin{bmatrix} b_{11}(t) & b_{12}(t) & 0 \\ b_{21} (t) & b_{22}(t) & 0\\ 0& 0 &0 \\ \end{bmatrix},
\]
which yields claim (iii). 

Let us denote by $\varepsilon:= ||\mathfrak{b}||_{C^{0}}\leq \sup_{t \in [0, T]} \sup_{i, j} |b_{ij} (t)|$. Observe first that the $C^{0}$ norm of $\tilde{W}_{\kappa}$ is bounded by $2 \varepsilon \kappa$. Indeed,
\[
|\partial_{x} \psi|=|-b_{11}(t) f'_{\kappa}(x) f_{\kappa}(y)+b_{21}(t) f'_{\kappa}(x) f_{\kappa}(x) | \leq  2 \varepsilon \kappa,
\]
and similarly for $|\partial_{y} \psi|$, this proves inequality \eqref{eq:W0}. Furthermore, the derivatives of the components of $\tilde{W}_{\kappa}$ with respect to either $x$ or $y$ are sums of up to four terms, each of which can be bounded as
\[
\varepsilon |(f'_{\kappa})^{2}+ f''_{\kappa} f_{\kappa}| \leq 101 \varepsilon.
\]
from which we deduce inequality \eqref{eq:Wxy}. Finally, the derivatives with respect to $t$ are of the form
\[
b'(t) f'_{\kappa} f_{\kappa},
\]
that is, of order $b'(t) \kappa$, and inequality \eqref{eq:Wt} follows. \end{proof}

\paragraph{\textbf{Step 2: controlled flow box chart}}

This second step ensures the existence of a volume preserving coordinate chart in a tubular neighborhood of the flow trajectory whose derivatives are suitably controlled. In what follows, $C_{X}, c_{X}$ will be two constants bounding respectively from above and below the norm of $\Gamma'(t)=X(\phi^{t}_{X})(p)$. We also set
\[
r(t):=\mu_{\phi^{t}_{X}(p)}(v_1(t), v_2(t), X(\phi^{t}_{X}(p))).
\] 

Note that we can take $C_{X}=||X||_{C^{0}(M)}$, but $c_{X}$ (and also $r(t)$) can be arbitrarily small (though never zero) if $X$ has zeroes. 

\begin{prop}\label{prop:diffeo}
For $\kappa_0$ small enough there is an embedding
\[
\Phi: \DD_{\kappa_0} \times [0, T] \to M
\]
(where $\DD_{\kappa_0}$ is the disk of radius $\kappa_0$) such that 
\begin{enumerate}
\item Pulling back the volume form yields
\[
\Phi^{*} \mu=r(t)dx \wedge dy \wedge dt.
\]
\item Its differential satisfies
\[
\frac{\partial \Phi}{\partial x}(0, 0, t)=v_1(t),\,\, \frac{\partial \Phi}{\partial y}(0, 0, t)=v_2(t),\,\,
\frac{\partial \Phi}{\partial t}(0,0,t)=X(\phi^{t}_{X}(p)),
\]
and thus, for any $\kappa<\kappa_0$
\[
\text{min}(c_{X}, 1/C_{\tau})-O(\kappa)\leq ||D \Phi||_{C^{0}(\DD_{\kappa}\times [0, T])} \leq \text{max}(C_{X}, C_{\tau})+O(\kappa)\,.
\]
\item Denote by \[\Theta=(\Theta^{x}, \Theta^{y}, \Theta^{t}): \Phi( \DD_{\kappa} \times [0, T]) \to \DD_{\kappa} \times [0, T]
\]
the inverse of $\Phi$, then
\[
||D \Theta^{x}||_{C^{0}} \leq C_{\tau}+O(\kappa) ,\,\, ||D \Theta^{y}||_{C^{0}} \leq C_{\tau}+O(\kappa) ,\,\, ||D \Theta^{t}||_{C^{0}} \leq \frac{1}{c_{X}}+O(\kappa).
\]
\end{enumerate}

\end{prop}

The remaining of step 2 proves this proposition. First define the map
\[
\tilde{\Phi}: \DD_{\tilde \kappa} \times [0, T] \to M
\]
\[
\tilde{\Phi}(\tilde{x}, \tilde{y}, t)= \exp_{\phi^{t}_{X}(p)}(\tilde{x} v_1(t)+\tilde{y} v_2(t)),
\]
where $\exp$ is the Riemannian exponential map and $\tilde{\kappa}$ is a radius to be fixed later. It is easy to check that the differential of $\tilde{\Phi}$ at any point $(0, 0, t)$ is
\[
\frac{\partial \tilde{\Phi}}{\partial \tilde{x}}(0, 0, t)=v_1(t),\,\, \frac{\partial \tilde{\Phi}}{\partial \tilde{y}}(0, 0, t)=v_2(t),\,\,
\frac{\partial \tilde{\Phi}}{\partial t}(0,0,t)=X(\phi^{t}_{X}(p)).
\]
So $\tilde{\Phi}$ is a diffeo into its image for all $\tilde{\kappa}$ small enough and, since it is $C^{2}$, we have the bounds
\[
\text{min}(c_{X}, 1/C_{\tau})-O(\tilde{\kappa}) \leq |D_{(\tilde{x}, \tilde{y}, t)} \tilde{\Phi}| \leq  \text{max}(C_{X}, C_{\tau})+ O(\tilde{\kappa}).
\]
(note that $\tilde{\Phi}$ does not depend on $\tilde{\kappa}$). This diffeomorphism does not satisfy in general the volume-preserving condition (i) in the proposition. The volume element $\tilde{\Phi}^{*} \mu$ in these coordinates is $h(\tilde{x}, \tilde{y}, t) d\tilde{x} \wedge d\tilde{y} \wedge dt$, where 
\[
h(\tilde{x}, \tilde{y}, t)=\mu\bigg(\frac{\partial \tilde{\Phi}}{\partial \tilde{x}}, \frac{\partial \tilde{\Phi}}{\partial \tilde{y}}, \frac{\partial \tilde{\Phi}}{\partial t}\bigg),
\]
in particular $h(0, 0, t)=r(t)$ and
\[
r(t)-O(\tilde{\kappa})\leq h(\tilde{x}, \tilde{y}, t) \leq r(t)+O(\tilde{\kappa}).
\]
Let us define
\[
\rho(\tilde{x}, \tilde{y}, t):=\frac{h(\tilde{x}, \tilde{y}, t)}{r(t)}.
\]
To construct the desired diffeomorphism, for $\kappa_0< \tilde{\kappa}$ small enough we perform a further change of variables on each disk $\DD_{\kappa_0}\times\{t\}$,
$\tilde{x}(x, y, t), \,\, \tilde{y}(x, y, t)$ so that $$dx \wedge dy=\rho(\tilde{x}(x, y, t), \tilde{y}(x, y, t), t) d\tilde{x} \wedge d\tilde{y}.$$ We denote  by $\Psi$ the resulting map
\[
\Psi(x, y, t)=(\tilde{x}(x, y, t), \tilde{y}(x, y, t), t).
\]

Since $\rho(0, 0, t)=1$, we can choose $\Psi$ to be the identity to first order. Thus, the composition $\Phi=\tilde{\Phi} \circ \Psi$ satisfies
\[
D_{(0, 0, t)} \Phi=D_{(0, 0, t)} \tilde{\Phi}=\begin{bmatrix}
    v_{1}(t) & v_{2}(t) & X(\phi^{t}_{X}(p)),
\end{bmatrix}
\]
and, for any $(x, y) \in \{ x^2+y^2<\kappa^2 \leq \kappa_{0}^2\}$:
\[
\text{min}(c_X, 1/C_{\tau})-O(\kappa) \leq |D_{(x, y, t)} \Phi| \leq  \text{max}(C_X, C_{\tau})+ O(\kappa),
\]
as claimed in (ii).

To see this more precisely, let us apply Moser's trick and define, for $s \in [0, 1]$,
\[
\omega_{s}(x, y, t):=(s \rho(x, y, t)+(1-s)) d x \wedge d y,
\]
and let $\lambda_{s}$ be a primitive of $\omega_s$, for example
\[
\lambda_{s}:=dy \int_{0}^{x} (s \rho(q, y, t)+(1-s))\,dq.
\]
Let $v_{s}(x, y, t)$ be the vector field satisfying
\[
i_{v_{s}} \omega_{s}= -\frac{\partial}{\partial s} \lambda_{s}=d y \int_{0}^{x} dq (1-\rho(q, y, t)),
\]
that is,
\[
v_{s}(x, y, t)=\frac{\int_{0}^{x} (\rho(q, y, t)-1) dq}{s \rho(x, y, t)+1-s} \frac{\partial}{\partial x}.
\]
For each fixed $t \in [0, T]$, the non-autonomous flow $\phi_{s}(x, y, t)$ of $v_{s}$ at time $s=1$ provides the desired change of variables:
\[
(\tilde{x}, \tilde{y})=\phi_{1}(x, y, t)=(\tilde{x}(x, y, t), y)
\]
Set $\Psi_{s}(x, y, t)=(\phi_{s}(x, y, t), t)$ and define $\Psi$ to be the map $\Psi_{1}$:
\[
(\tilde{x}, \tilde{y}, t)=\Psi_{1}(x,y, t):=(\phi_{1}(x, y, t), t).
\]
which is the flow at time $s=1$ of the non-autonomous vector field on $\DD_{\kappa} \times [0, T]$
\[V_{s}(x, y, t)=\frac{\int_{0}^{x} (\rho(q, y, t)-1) dq}{s \rho(x, y, t)+1-s} \frac{\partial}{\partial x}.\]
Now the matrix $D_{(0, 0, t)} \Psi$ is the time $s=1$ solution of the ODE
\[
\frac{d}{ds} D_{(0, 0, t)} \Psi_{s} =D_{\Psi_{s}(0, 0, t)} V_{s} \cdot D_{(0, 0, t)} \Psi_{s},\,\,\,D_{(0, 0, t)} \Psi_{0}=\text{Id}.\]
Since both $V_s$ and its first derivatives vanish at $(0, 0, t)$, $\Psi(0, 0, t)=(0, 0, t)$ and $D_{(0, 0, t)} \Psi$ is the identity matrix, so $\Psi$ is the identity to first order as claimed. 

Now let $\Theta=(\Theta^x, \Theta^y, \Theta^t)$ denote the inverse map, which is defined in a small enough tubular neighbor of the orbit. Since $D_{\Phi(x, y, t)} \Theta \circ D_{(x, y, t)} \Phi=\text{Id}$, at the point $(0, 0, t)=\Phi^{-1}(\phi^{t}_{X}(p))$ we have
\[
D_{\phi^{t}_{X}(p)} \Theta^{x}(v_{1}(t))=1,\,\,D_{\phi^{t}_{X}(p)} \Theta^{x}(v_{2}(t))=0,\,\,D_{\phi^{t}_{X}(p)} \Theta^{x}(X(\phi^{t}_{X}(p)))=0.
\]
Since $\{v_1(t), v_2(t), X(\phi^{t}_{X}(p))\}$ is a basis and $|v_{i}(t)|\geq \frac{1}{C_{\tau}}$, we conclude that 
\[
|D_{q} \Theta^x| \leq C_{\tau}+O(\kappa),
\]
for any point $q \in \Phi(\DD_{\kappa} \times [0, T])$. We argue analogously for $\Theta^y$. As for $\Theta^t$, we have
\[
D_{\phi^{t}_{X}(p)} \Theta^{t}(v_{1}(t))=0,\,\,D_{\phi^{t}_{X}(p)} \Theta^{t}(v_{2}(t))=0,\,\,D_{\phi^{t}_{X}(p)} \Theta^{t}(X(\phi^{t}_{X}(p)))=1,
\]
and thus 
\[
|D_{\phi^{t}_{X}(p)} \Theta^{t}|=\frac{1}{|X(\phi^{t}_{X}(p))|},
\]
from which we get the bound $|D_{q} \Theta^{t}| \leq \frac{1}{c_X}+O(\kappa)$.\\

\paragraph{\textbf{Step 3: Construction of the perturbation}}

Let $\Phi: \DD_{\kappa_0}\times [0, T] \to M$ be the map constructed in Step 2, and let $\tilde{W}_{\kappa}$ be the vector field constructed in Step 1 for the $\mathfrak{sl}(2, \RR)$ matrix $\mathfrak{b}(t):=\mathfrak{a}_{X}(t)-\mathfrak{c}(t)$. The claim is that for $\kappa< \kappa_0$ small enough the vector field $Y=X+\Phi_{*} \tilde{W}_{\kappa}$ satisfies all the properties in the statement of Lemma \ref{lem:franks2}.

Let us denote by $\psi^{t}_{X}$ the flow of the vector field $X$ in the $(x, y, t)$ coordinates defined by $\Phi$, that is
\[
\psi^{t}_{X}:= \Phi^{-1} \circ \phi_{X}^{t} \circ \Phi,
\]
and by $\tilde{X}$ the corresponding vector field in these coordinates, $\tilde{X}=\Phi^{-1}_{*} X=\Theta_{*} X$.

Since $(0, 0, t)$ is a trajectory of $\tilde{X}$, we have
\[D_{(0, 0, 0)}\psi_{X}^{t}=\begin{bmatrix}G_{11}(t) & G_{12}(t) & 0 \\ G_{21} (t) & G_{22}(t) &0 \\ H_1(t) & H_2(t) & 1\end{bmatrix}\,.\]
for some functions $G_{ij}(t)$ and $H_i(t)$. Denote the upper $2 \times 2$ matrix by $G(t)$. The relation
\[
\frac{d}{dt} D_{(0, 0, 0)} \psi_{X}^{t}= D_{(0, 0, t)} \tilde{X} \circ D_{(0, 0, 0)} \psi_{X}^{t}
\] 
implies
\[D_{(0, 0, t)} \tilde{X}=\begin{bmatrix} g_{11}(t) & g_{12}(t) & 0 \\ g_{21} (t) & g_{22}(t) &0 \\ h_1(t)& h_2(t) & 0\end{bmatrix},\]
where (denoting by $g(t)$ the upper $2\times2$ matrix) \[\frac{d}{dt} G(t)=g(t) \circ G(t)\,.\]
It is easy to check that, precisely, we have 
\[
A_{X}(t)=\begin{bmatrix} G_{11}(t) & G_{12}(t)  \\  G_{21} (t) & G_{22}(t) \end{bmatrix},
\]
and 
\[
\mathfrak{a}_{X}(t)=\begin{bmatrix} g_{11}(t) & g_{12}(t)  \\ g_{21} (t) & g_{22}(t) \end{bmatrix}.
\]
Indeed, by definition,
\[
A_{X}(t)=\tau_{\phi^{t}_{X}(p)} \circ D_{(0, 0, t)} \Phi \circ D_{(0, 0, 0)} \psi^{t}_{X} \circ D_{p} \Phi^{-1} \circ \tau^{-1}_{p}.
\]
On the other hand, from Proposition \ref{prop:diffeo} and the definition of the vectors $v_i(t)$ we obtain
\[
\tau_{\phi^{t}_{X}(p)} \circ D_{(0, 0, t)} \Phi \bigg(\frac{\partial}{\partial x}\bigg)=\tau_{\phi^{t}_{X}(p)}(v_{1}(t))=\frac{\partial}{\partial x},
\]
and similarly
\[
\tau_{\phi^{t}_{X}(p)} \circ D_{(0, 0, t)} \Phi \bigg(\frac{\partial}{\partial y}\bigg)=\frac{\partial}{\partial y},\,\,
\tau_{\phi^{t}_{X}(p)} \circ D_{(0, 0, t)} \Phi \bigg(\frac{\partial}{\partial t}\bigg)=0\,.
\]
As for the inverse, we have
\[
 D_{p} \Phi^{-1} \circ \tau^{-1}_{p} \bigg( \frac{\partial}{\partial x} \bigg)=\frac{\partial}{\partial x},\,\,\,
 D_{p} \Phi^{-1} \circ \tau^{-1}_{p} \bigg( \frac{\partial}{\partial y} \bigg)=\frac{\partial}{\partial y},
\]
so we get, as claimed, that
\[
A_{X}(t)=\begin{bmatrix} G_{11}(t) & G_{12}(t)  \\ G_{21} (t) & G_{22}(t)\end{bmatrix},
\]
and
\[
\mathfrak{a}_{X}(t)=\bigg(\frac{d}{dt} A_{X}(t)\bigg) \circ A_{X}^{-1}(t)=\bigg(\frac{d}{dt} G(t)\bigg) \circ G^{-1}(t)=g(t).
\]

Now set
\[
\mathfrak{b}(t):=\mathfrak{a}_{X}(t)-\mathfrak{c}(t)=\begin{bmatrix} b_{11}(t) & b_{12}(t)  \\ b_{21} (t) & b_{22}(t) \end{bmatrix}.
\]
For a $\kappa$ to be fixed later at convenience, let $\tilde{W}_{\kappa}(x, y, t)$ be the vector field from Step 1. Define the vector field $\tilde{Y}$ as
\[
\tilde{Y}=\tilde{X}+\tilde{W}_{\kappa}
\]
and the corresponding vector field $Y$ as $Y=\Phi_{*} \tilde{Y}$. Recall that $\tilde{W}_{\kappa}$ is exact with respect to any volume form $\rho(t) dx \wedge dy \wedge dz$ (Proposition \ref{prop:localW}, (ii)), so in particular with respect to $\Phi^{*} \mu=r(t) dx \wedge dy \wedge dz$, and thus $Y \in \mathfrak{X}_{\mu}^{0}$ (this proves the first claim in the statement of Lemma \ref{lem:franks2}).

Furthermore, because $\tilde{W}_{\kappa}(0, 0, t)=0$, we have $\phi^{t}_{X}(p)=\phi^{t}_{Y}(p)$ (this proves item (iii) in the stament of the Lemma) and $(0, 0, t)$ is an orbit of $\tilde{Y}$. Thus arguing as in the case of $X$ above and taking into account that (Proposition \ref{prop:localW}, (iii))
\[
D_{(0, 0, t)} \tilde{W}_{\kappa}=\begin{bmatrix} b_{11}(t) & b_{12}(t) & 0 \\ b_{21} (t) & b_{22}(t) & 0\\ 0& 0 &0 \\ \end{bmatrix},
\]
we get
\[
\bigg(\frac{d}{dt} A_{Y}(t)\bigg) \circ A_{Y}^{-1}(t)=\mathfrak{a}_{X}(t)+\mathfrak{b}(t)=\mathfrak{c}(t),
\]
where
\[
A_{Y}(t):=\tau_{\phi^{t}_{Y}(p)} \circ D_{p} \phi^{t}_{Y} \circ \tau^{-1}_{p}.
\]
This proves item (iv) in Lemma \ref{lem:franks2}. 

Now, $\tilde{W}_{\kappa}$ has support inside $\mathbb{D}_{\kappa} \times [0, T]$ (Proposition \ref{prop:localW}, (i)), so $X=Y$ outside $\Phi(\DD_{\kappa}\times [0, T])$: choosing $\kappa$ small enough yields item (i) in Lemma \ref{lem:franks2}. 

We are left to prove what is the most important item in the Lemma for our purposes, item (ii)
\[
||X-Y||_{C^1(M)}=||W||_{C^1(M)} \leq K \varepsilon,
\]
where $K$ is a constant depending only on $C_{\tau}$ and $||X||_{C^{0}}$. This follows from the bounds in Proposition \ref{prop:localW},(iv), together with the bounds in the derivatives of $\Phi$ and its inverse given by items (ii) and (iii) in Proposition \ref{prop:diffeo}. Here, the property that $\tilde{W}_{\kappa}$ and its derivative with respect to $t$ are of order $\kappa$ (bounds \eqref{eq:W0} and \eqref{eq:Wt}) is crucial: otherwise some derivatives of $W$ could be arbitrarily large when $X$ has zeroes. 

To see this more precisely, fix an arbitrary coordinate chart $\{u^i\}_{i=1}^{3}$ on a neighborhood $U$ of a point $q \in \Gamma$. We identify the open set $U$ with the corresponding open set in $\RR^3$ in the $u^{i}$ coordinates, and with a slight abuse of notation still denote by $\Phi$ the map $\Phi: \DD_{\kappa} \times [0, T] \cap \Phi^{-1}(U) \to U \subset \RR^3$. To alleviate the expressions below, let us also denote the coordinates $(x, y, t)$ in $\DD_{\kappa} \times [0, T]$  by $\{x^{\mu}\}_{\mu=1}^{3}$.  Then we can write (here we sum over repeated indices)
\[
W(\Phi(x, y, t))=\tilde{W}_{\kappa}^{\mu}(x, y, t)\frac{\partial \Phi^{i}(x, y,t)}{\partial x^{\mu} } \frac{\partial}{\partial u^{i}},
\]
so by Proposition \ref{prop:diffeo}, (ii), and Proposition \ref{prop:localW}, (iv), we get, taking $\kappa$ small enough,
\[
||W||_{C^{0}(M)} \leq 2 \kappa ||\mathfrak{b}(t)||_{C^{0}([0, T])} (C_{X}+O(\kappa))  \leq 3 \kappa \varepsilon||X||_{C^{0}(M)}.
\]
where in the last inequality we have used that $||\mathfrak{b}(t)||_{C^{0}} \leq \varepsilon$ and $C_{X}=||X||_{C^{0}(M)}$.
As for the first derivatives of $W$, we have (recall that $\Theta=\Phi^{-1}$)
\[
\frac{\partial W^{i}}{\partial u^{j}}= \frac{\partial \tilde{W}_{\kappa}^{\mu}}{\partial x^{\nu}} \frac{\partial \Theta^{\nu}}{\partial u^{j}} \frac{\partial \Phi^{i}}{\partial x^{\mu}}+\tilde{W}_{\kappa}^{\mu} \frac{\partial \Theta^{\nu}}{\partial u^{j}} \frac{\partial^{2} \Phi^{i}}{\partial x^{\nu} \partial x^{\mu} }  .
\]
Let us focus first on the terms with first derivatives of $\tilde{W}_{\kappa}$.

For $\nu=1,2$, we have, using bounds \eqref{eq:Wxy} in Proposition \ref{prop:localW} and the bounds in the derivatives of $\Phi$, $\Theta^{x}$ and $\Theta^{y}$ in Proposition \ref{prop:diffeo} (notice there are six such terms when summing over repeated indices)
\begin{equation}\label{eq:W2xy}
(\nu=1, 2)\,\,\,\,\,\bigg|\frac{\partial \tilde{W}_{\kappa}^{\mu}}{\partial x^{\nu}} \frac{\partial \Theta^{\nu}}{\partial u^{j}} \frac{\partial \Phi^{i}}{\partial x^{\mu}}\bigg| \leq 6 \cdot 1000 \varepsilon \cdot (C_{\tau}+O(\kappa))\cdot (\text{max}(C_{\tau}, C_{X})+O(\kappa)) \leq \tilde{K} \varepsilon,
\end{equation}
for some constant $\tilde{K}$ depending only on $C_{\tau}$ and $C_{X}=||X||_{C^{0}(M)}$, and $\kappa$ small enough.  
For $\nu=3$, that is $x^{\nu}=t$, using the bound \eqref{eq:Wt} we get
\begin{equation}\label{eq:W2t}
\bigg|\frac{\partial \tilde{W}_{\kappa}^{\mu}}{\partial t} \frac{\partial \Theta^{t}}{\partial u^{j}} \frac{\partial \Phi^{i}}{\partial x^{\mu}}\bigg| \leq 3 \cdot 2 \kappa ||\mathfrak{b}'(t)|| \cdot (1/c_{X}+O(\kappa))\cdot (\text{max}(C_{\tau}, C_{X})+O(\kappa)).
\end{equation}
Thus for any $\kappa<\kappa_1$ small enough ($\kappa_1$ depending on $\varepsilon$ and how large $1/c_{X}$ and $||\mathfrak{b}'(t)||$ are) we get
\begin{equation}\label{eq:W22}
\bigg|\frac{\partial \tilde{W}_{\kappa}^{\mu}}{\partial t} \frac{\partial \Theta^{t}}{\partial u^{j}} \frac{\partial \Phi^{i}}{\partial x^{\mu}}\bigg| \leq \varepsilon.
\end{equation}
\begin{remark}
Observe that the derivative of $\partial_{t} \tilde{W}_{\kappa}$ being bounded by $\kappa$ (bound \eqref{eq:Wt}) is the crucial point here, as it counters the fact that the derivatives of $\Theta^{t}$ are bounded by $1/c_{X}$, which can be arbitrarily large (still always bounded for any finite $T$) when $X$ has zeroes.
\end{remark}
It remains to bound by some $K \varepsilon$ the terms with second derivatives of $\Phi$. As with the $\nu=3$ term above, the key is that $\tilde{W}$ is also $O(\kappa)$. More precisely, let $C_{2}(\Gamma)$ be a constant bounding the second derivatives of $\Phi$ in $\DD_{\kappa_0}\times [0, T]$. This constant does not depend only on $||X||_{C^{0}(M)}$ and $C_{\tau}$, but since $|\tilde{W}_{\kappa}| \leq 2 \varepsilon \kappa$ we have
\[
\bigg|\tilde{W}_{\kappa}^{\mu} \frac{\partial \Theta^{\nu}}{\partial u^{j}} \frac{\partial^{2} \Phi^{i}}{\partial x^{\nu} \partial x^{\mu} }\bigg| \leq 9 \cdot 2 \varepsilon \kappa \cdot (1/c_{X}+O(\kappa)) \cdot C_{2}(\Gamma).
\]
Again for any $\kappa$ small enough (depending on $c_{X}$ and $C_{2}(\Gamma)$), we get
\[
\bigg|\tilde{W}_{\kappa}^{\mu} \frac{\partial \Theta^{\nu}}{\partial u^{j}} \frac{\partial^{2} \Phi^{i}}{\partial x^{\nu} \partial x^{\mu} }\bigg| \leq \varepsilon.
\]
Putting together equations \eqref{eq:W2xy}, \eqref{eq:W2t} and \eqref{eq:W22} we get
\[
\bigg|\frac{\partial W^{i}}{\partial u^{j}}\bigg| \leq \tilde{K} \varepsilon+\varepsilon+\varepsilon \leq K \varepsilon,
\]
with $K$ depending only on $C_{\tau}$ and $||X||_{C^{0}(M)}$. This and the $C^{0}$ bound on $W$ yields
\[
||W||_{C^{1}(M)} \leq K \varepsilon,
\]
as claimed. \end{proof}

\bibliographystyle{hplain}
\bibliography{standard_bib}

\end{document}